\documentclass[12pt,reqno]{amsart}

\usepackage[vmargin=1in,hmargin=1.2in]{geometry}
\usepackage{graphicx, latexsym, graphics}
\usepackage{amsfonts, amssymb, amsmath, amsthm, mathtools, mathrsfs}
\usepackage{algorithmic, algorithm}
\usepackage{tikz}
\usetikzlibrary{matrix,arrows}

\newcommand{\NN}{\mathbb N}

\newcommand{\CC}{\mathbb C}
\newcommand{\RR}{\mathbb R}
\newcommand{\ZZ}{\mathbb Z}

\newcommand{\DD}{\mathcal D}
\newcommand{\SSS}{\mathcal S}

\theoremstyle{plain}
\newtheorem{theorem}{Theorem}[section]

\newtheorem{lemma}[theorem]{Lemma}
\newtheorem{corollary}[theorem]{Corollary}

\theoremstyle{remark}
\newtheorem{remark}[theorem]{Remark}

\theoremstyle{definition}

\allowdisplaybreaks

\numberwithin{equation}{section}

\newcommand{\beq}{\begin{eqnarray}}
\newcommand{\eeq}{\end{eqnarray}}

\newcommand{\beqs}{\begin{eqnarray*}}
\newcommand{\eeqs}{\end{eqnarray*}}

\newcommand{\supp}{\operatorname{supp}}

\def\Modsq#1{{\mathcal{M}[{#1}]}}

\begin{document}

\title[Modulation spaces associated to tensor products of amalgam spaces]{Modulation spaces associated to tensor products of amalgam spaces}

\author[H.~G.~Feichtinger]{Hans G. Feichtinger}
\address{H. G. Feichtinger, Faculty of Mathematics, University of Vienna, Oskar-Morgenstern-Platz 1, A-1090 Wien, Austria}
\email{hans.feichtinger@univie.ac.at}

\author[S. Pilipovi\'{c}]{Stevan Pilipovi\'{c}}
\address{S. Pilipovi\'{c}, Department of Mathematics and Informatics, University of Novi Sad, Trg Dositeja Obradovi\'ca 4, 21000 Novi Sad, Serbia}
\email{stevan.pilipovic@dmi.uns.ac.rs}

\author[B. Prangoski]{Bojan Prangoski}
\address{B. Prangoski, Faculty of Mechanical Engineering, University ``Ss. Cyril and Methodius'',
Karpos II bb, 1000 Skopje, Macedonia}
\email{bprangoski@yahoo.com}

\subjclass[2010]{Primary 46F05. Secondary 46H25; 46E10; 46F12; 81S30}
\keywords{Wiener amalgam spaces; Modulation spaces; Translation and modulation invariant Banach space}

\maketitle

\begin{abstract}
We identify the modulation spaces associated to tensor products of amalgam spaces having a large class of Banach spaces as their local component. As consequences of the main results, we describe the modulation spaces associated to tensor products of various $L^p$ spaces.
\end{abstract}
\maketitle

\section{Introduction}
\nocite{fegr89,fegr92-1,fe03-1,fe06,fegr85}  
The modulation spaces were introduced by the first author
in a technical report prepared 1983 and published in 2002 (\cite{fe03-1}).
Subsequently the theory was further developed in the context
of {\it coorbit spaces}
by him and Gr\"ochenig in \cite{fegr89} and \cite{fegr92-1}. We
refer to the monograph \cite{Grochenig} for an overview. Nowadays, they are studied by many authors and are widely accepted as an indispensable tool in time-frequency analysis.\\
\indent A natural generalisation of the classical modulation spaces $M^{p,q}_{\eta}$ was introduced in \cite{dppv-3} in the following way. Let $X$ be a Banach space of tempered distributions laying in between $\SSS(\RR^{2d})$ and $\SSS'(\RR^{2d})$, invariant under translations and modulations and satisfying certain technical conditions (see Section \ref{tmib-mod-spa}); so called translation-modulation invariant Banach spaces. The modulation spaces $\Modsq{X}$ (denoted as $\mathcal{M}^X$ in \cite{dppv-3}) associated to $X$ consists of all tempered distributions in $\SSS'(\RR^d)$ whose image under the short-time Fourier transform belongs to $X$. Taking $X$ to be the weighted mixed-norm Lebesgue space $L^{p,q}_{\eta}$, $\Modsq{X}$ reduces to $M^{p,q}_{\eta}$. These generalised modulation spaces satisfy all of the important basic properties as the classical ones (see \cite{dppv-3}). Unlike $L^{p,q}_{\eta}$, the space $X$ does not need to be solid (in the sense of \cite{fei84}) and this generalised framework allows one to consider a wide variety of modulation spaces $\Modsq{X}$. An interesting choice for $X$ beyond $L^{p,q}_{\eta}$ is to take a completed $\pi$-tensor product (also called the projective tensor product \cite{ryan}) of two $L^p$ spaces on $\RR^d$, i.e. $X=L^{p_1}(\RR^d)\hat{\otimes}_{\pi} L^{p_2}(\RR^d)$ (these are rarely solid; see \cite[Remark 3.10]{dppv-3}). In fact, \cite[Proposition 5.1 (ii)]{dppv-3} asserts that $\Modsq{X}\subseteq L^1(\RR^d)$ when $X=L^{p_1}\hat{\otimes}_{\pi} L^{p_2}$ with $1\leq p_1\leq p_2\leq 2$. The main goal of this article is to identify $\Modsq{X}$ when $X$ is a tensor product. As a consequence of one of our main results we can explicitly describe $\Modsq{X}$ when $X=L^{p_1}\hat{\otimes}_{\pi} L^{p_2}$ with $1\leq p_1\leq p_2\leq 2$ (see Corollary \ref{cor12} and Remark \ref{rem-for-lpp}):
$$
\Modsq{L^{p_1}\hat{\otimes}_{\pi} L^{p_2}}=W(\mathcal F L^{p_2},L^1)=\mathcal{F}M^{p_2,1},\quad \mbox{for all}\,\, 1\leq p_1\leq p_2\leq 2.
$$
\indent The amalgam spaces \cite{feich,fe81} (often called Wiener amalgam spaces) will play a key role in our analysis. As it turns out, it is easier to identify $\Modsq{X}$ when $X$ is a tensor product of two amalgam type spaces and then deduce from it various results concerning tensor products of $L^p$ spaces. In the first main result of the article (Theorem \ref{1.21}) we prove that $\Modsq{X}$ is an amalgam space of the form $W(E,L^1_{\eta})$ when $X$ is the completed $\pi$-tensor product of a large class of amalgam spaces; here $E\subseteq \SSS'(\RR^d)$ is a Banach space of tempered distributions (which satisfies the same general properties as $X$). This is interesting in itself as it partially answers the
converse  
question: when $W(E,L^p_{\eta})$ can be regarded as a (generalised) modulation space? Theorem \ref{1.21} claims this is always the case when $p=1$ for any so called translation-modulation invariant Banach space $E$. When $E=\mathcal{F}L^r$, the amalgam space $W(\mathcal{F}L^r, L^p)$ is exactly the Fourier image of the classical modulation space $M^{r,p}$ (cf. Lemma \ref{ident-mod-amalg}). 
For $p=1$ these amalgam spaces can be considered as a modulation spaces as well.\\
\indent We also consider $\Modsq{X}$ when $X$ is the completed $\epsilon$-tensor product (also called the injective tensor product \cite{ryan}) of a class of two amalgam spaces. Our second main result (Theorem \ref{meinproposition1}) shows that in this case $\Modsq{X}$ is an amalgam space of the form $W(E,L^{\infty}_{\eta,0})$, where $L^{\infty}_{\eta,0}$ denotes that the elements of the amalgam space vanish at infinity with weight $\eta$.\\
\indent The article is organised as follows. In Section \ref{notation} we recall the definitions and basic properties of the translation-modulation invariant Banach spaces and the generalised modulation spaces. Section \ref{amalgam-spaces-s} is devoted to the amalgam spaces and their properties, especially when they are constructed from translation-modulation invariant Banach spaces. Sections \ref{spi} and \ref{sepsilon} contain the main results: Theorem \ref{1.21} and Theorem \ref{meinproposition1}. In the last section we collect a number of important consequences (and special cases) of the two main results.

\section{Notation and Preliminaries}\label{notation}

A positive measurable function $\eta:\RR^d\rightarrow (0,\infty)$ is said to be a {\it polynomially bounded weight} on $\RR^d$ if there are $C,\tau>0$ such that
\beqs
\eta(x+y)\leq C\eta(x)(1+|y|)^{\tau},\quad \mbox{for all}\,\, x,y\in\RR^d.
\eeqs
Such functions $\eta$ are also called \textit{moderate} with respect to the Beurling weight $(1+|y|)^{\tau}$, $\tau\geq 0$. Given a polynomially bounded weight $\eta$, we denote by $L^p_{\eta}(\RR^d)$, $1\leq p\leq \infty$,
the weighted $L^p$ space of measurable functions $f$ with norm $\|f\|_{L^p_{\eta}}=\|f\eta\|_{L^p}$. When $p=\infty$, we also consider the closed subspace $L^{\infty}_{\eta,0}(\RR^d)$ of $L^{\infty}_{\eta}(\RR^d)$
 consisting of all $f\in L^{\infty}_{\eta}(\RR^d)$ which satisfy the following:
 for every $\varepsilon>0$ there exists a compact set $K\subseteq \RR^d$
 such that $\|f\eta\|_{L^{\infty}(\RR^d\backslash K)}\leq \varepsilon$. Additionally,
 we denote by $\mathcal{C}_{\eta,0}(\RR^d)$ the space
 $L^{\infty}_{\eta,0}(\RR^d)\cap \mathcal{C}(\RR^d)$;
 it is a closed subspace of $L^{\infty}_{\eta,0}(\RR^d)$
  and we equip it with the induced norm $\|\cdot\|_{L^{\infty}_{\eta}}$. When $\eta=1$, we employ the notations $L^{\infty}_0(\RR^d)$ and $\mathcal{C}_0(\RR^d)$ instead of $L^{\infty}_{\eta,0}(\RR^d)$ and $\mathcal{C}_{\eta,0}(\RR^d)$ (i.e., $\mathcal{C}_0(\RR^d)$ is the space of all continuous
  functions which vanish at infinity with supremum norm). Although we will
  never do so, there is no loss of generality in assuming $\eta$ to be continuous when working with either one of these space because one can always find an equivalent continuous polynomially bounded weight $\tilde{\eta}$, i.e.\ satisfying $C'^{-1}\tilde{\eta}(x)\leq \eta(x)\leq C'\tilde{\eta}(x)$, $\forall x\in\RR^d$, for some $C'>0$, defining the same weighted spaces. \\
\indent We use the following version of the Fourier transform:  $\mathcal{F}f(\xi)=\int_{\RR^d}e^{-2\pi i x\xi}f(x)dx$, $f\in L^1(\RR^d)$. The translation and modulation operators $T_x$, $x\in\RR^d$, and $M_{\xi}$, $\xi\in\RR^d$, are defined as follows: $T_xf(t)=f(t-x)$ and $M_{\xi}f(t)=e^{2\pi i t\xi}f(t)$. The short-time Fourier transform (from now, always abbreviated as STFT) of $f\in\SSS'(\RR^d)$ with window $g\in\SSS(\RR^d)\backslash\{0\}$ is defined as
\beqs
V_gf(x,\xi)=\langle f,\overline{M_{\xi}T_x g}\rangle=(f,M_{\xi}T_x g),\quad x,\xi\in\RR^d,
\eeqs
where $(\cdot,\cdot)$ stands for the sesquilinear form induced by the inner product on $L^2(\RR^d)$. When $f$ is a measurable function with polynomial growth, we have
\beqs
V_gf(x,\xi)= \int_{\RR^d} e^{-2\pi i t\xi} f(t)\overline{g(t-x)}dt,\quad x,\xi\in\RR^d.
\eeqs
For any $f\in\SSS'(\RR^d)$ and $g\in\SSS(\RR^d)\backslash\{0\}$, $V_gf$ is a smooth function on $\RR^{2d}$; furthermore $V_g:\SSS'(\RR^d)\rightarrow \SSS'(\RR^{2d})$ is continuous and it restricts to a continuous operator $V_g:\SSS(\RR^d)\rightarrow \SSS(\RR^{2d})$. The adjoint $V_g^*$ of $V_g$ is a continuous operator from $\SSS(\RR^{2d})$ into $\SSS(\RR^d)$ given by
\beqs
V_g^*\Phi(t)=\int_{\RR^{2d}} e^{2\pi i \xi t}\Phi(x,\xi)g(t-x)dxd\xi,\quad \Phi\in\SSS(\RR^{2d}),
\eeqs
and it extends to a continuous operator $V_g^*:\SSS'(\RR^{2d})\rightarrow \SSS'(\RR^d)$. Furthermore, for any $g_1,g_2\in\SSS(\RR^d)\backslash\{0\}$, $V_{g_1}^*V_{g_2}=(g_1,g_2)\operatorname{Id}$ on $\SSS'(\RR^d)$.

\subsection{Translation-modulation invariant Banach spaces of distributions and their duals. Modulation spaces}\label{tmib-mod-spa}

Here and throughout the rest of the article, for any two locally convex spaces $X$ and $Y$, $\mathcal{L}(X,Y)$ stands for the space of continuous linear mappings from $X$ into $Y$ while $\mathcal{L}_b(X,Y)$ for this space equipped with the topology of uniform convergence on all bounded subsets of $X$. When $X$ and $Y$ are Banach spaces, we denote by $\|\cdot\|_{\mathcal{L}_b(X,Y)}$ the operator norm on the Banach space $\mathcal{L}_b(X,Y)$ induced by the norms $\|\cdot\|_X$ and $\|\cdot\|_Y$. If $X=Y$, we will often abbreviate these notations as $\mathcal{L}(X)$, $\mathcal{L}_b(X)$ and $\|\cdot\|_{\mathcal{L}_b(X)}$. Finally, $X\hookrightarrow Y$ will always mean that $X$ is continuously and densely included into $Y$.\\
\indent A Banach space $E$ is said to be a \textit{translation-modulation invariant Banach spaces of distributions} on $\RR^d$ (or, a TMIB space for short) if it satisfies the following conditions \cite{dppv-3} (cf. \cite{DPPV,dpv,pbb}).
\begin{itemize}
\item[(I)] $\mathcal{S}(\mathbb{R}^d)\hookrightarrow E\hookrightarrow \mathcal{S}'(\RR^d)$.
\item[(II)] $T_x(E)\subseteq E$ and $M_{\xi}(E)\subseteq E$, for all $x,\xi\in\RR^d$.
\item[(III)] There exist $\tau,C>0$ such that\footnote{$T_x$ and $M_{\xi}$ are continuous on $E$ because of (I), (II) and the closed graph theorem}
    \begin{equation*}
    \omega_E(x)=\|T_x\|_{\mathcal{L}_b(E)}\leq C (1+|x|)^{\tau}\quad \mbox{and}\quad \nu_E(\xi)=\|M_{-\xi}\|_{\mathcal{L}_b(E)}\leq C (1+|\xi|)^{\tau}.
    \end{equation*}
\end{itemize}
We recall the basic facts of TMIB spaces and refer to \cite[Section 3]{dppv-3} (see also \cite{dpv}) for the complete account. We start by pointing out that $E$ is always separable and $\omega_E$ and $\nu_E$ are measurable submultiplicative locally bounded functions which satisfy $\omega_E(0)=\nu_E(0)=1$. The space $E$ is a Banach module with respect to convolution over the Beurling (convolution) algebra $L^1_{\omega_E}$ and a Banach module with respect to multiplication over the Wiener-Beurling algebra $\mathcal{F}L^1_{\nu_E}$ \cite{beurling,lst} \footnote{sometimes denoted by $A_{\nu_E}$}; the multiplication operation on $\mathcal{F}L^1_{\nu_E}$ is defined via the Fourier transform and the convolution on $L^1_{\nu_E}$ (when $\nu_E$ is bounded from below, $\mathcal{F}L^1_{\nu_E}\subseteq \mathcal{C}(\RR^d)$ and the multiplication reduces to ordinary pointwise multiplication). To be precise, the convolution and multiplication operations on $\SSS(\RR^d)$ uniquely extend to continuous bilinear mappings $*:L^1_{\omega_E}\times E\rightarrow E$ and $\cdot :\mathcal{F}L^1_{\nu_E}\times E\rightarrow E$ such that
\beqs
\|f*e\|_E&\leq& \|f\|_{L^1_{\omega_E}}\|e\|_E,\quad \forall f\in L^1_{\omega_E},\, \forall e\in E,\\
\|g\cdot e\|_E&\leq& \|g\|_{\mathcal{F}L^1_{\nu_E}}\|e\|_E,\quad \forall g\in\mathcal{F}L^1_{\nu_E},\,\forall e\in E,
\eeqs
and $f_1*(f_2*e)=(f_1*f_2)*e$, for all $f_1,f_2\in L^1_{\omega_E}$, $e\in E$, and $g_1\cdot(g_2\cdot e)=(g_1\cdot g_2)\cdot e$, for all $g_1,g_2\in\mathcal{F}L^1_{\nu_E}$, $e\in E$.\\
\indent The Fourier image of $E$ (i.e. the associated Fourier space) $\mathcal{F}E=\{\mathcal{F}f\in\SSS'(\RR^d)|\, f\in E\}$ is again a TMIB space with norm $\|\mathcal{F}f\|_{\mathcal{F}E}=\|f\|_E$, $\mathcal{F}f\in\mathcal{F}E$, and
\beqs
\omega_{\mathcal{F}E}(x)=\check{\nu}_E(x)=\nu_E(-x)\quad \mbox{and}\quad \nu_{\mathcal{F}E}(\xi)=\omega_E(\xi).
\eeqs
Similarly, $\mathcal{F}^{-1}E=\{\mathcal{F}^{-1}f\in\SSS'(\RR^d)|\, f\in E\}$ with norm $\|\mathcal{F}^{-1}f\|_{\mathcal{F}^{-1}E}=\|f\|_E$ is also a TMIB space. If $f\mapsto \check{f}$ restricts to a well-defined isometry on $E$, then $\mathcal{F}E=\mathcal{F}^{-1}E$ and $\|\cdot\|_{\mathcal{F}E}=\|\cdot\|_{\mathcal{F}^{-1}E}$. This holds, for example, when $E=L^p_{\eta}(\RR^d)$, $1\leq p<\infty$, and $E=\mathcal{C}_{\eta,0}(\RR^d)$ in the case $\eta(x)=(1+|x|)^s$, $s\in\RR$.

\begin{remark}\label{equiv-norm-for-bou-gr-f}
If $E$ is a TMIB space such that $\omega_E$ is bounded from above then $e\mapsto \sup_{x\in\RR^d}\|T_x e\|_E$, $E\rightarrow [0,\infty)$, is a norm on $E$ equivalent to its original norm and with respect to this new norm the translation operators $T_x$ are isometries on $E$ for all $x\in\RR^d$.
\end{remark}

A Banach space $E$ is said to be a \textit{dual translation-modulation invariant Banach space of distributions} (a DTMIB space for short) if it is the strong dual of a TMIB space. If $E$ is a DTMIB space then $\SSS(\RR^d)\subseteq E\subseteq \SSS'(\RR^d)$ and the inclusions are continuous. Furthermore, the translation and modulation operators are continuous on $E$ and, setting $E=F'$ with $F$ a TMIB space, we always have $\omega_E=\check{\omega}_F$ and $\nu_E=\nu_F$, where, as before, $\omega_E(x)=\|T_x\|_{\mathcal{L}_b(E)}$ and $\nu_E(\xi)=\|M_{-\xi}\|_{\mathcal{L}_b(E)}$. However, in general, $\SSS(\RR^d)$ may fail to be dense in $E$ (a typical example is $E=L^{\infty}(\RR^d)$); but, when $E$ is the strong dual of a reflexive TMIB space, $E$ is also a TMIB space. Given a DTMIB space $E$, one can introduce convolution with elements of $L^1_{\omega_E}$ and multiplication with elements of $\mathcal{F}L^1_{\nu_E}$ by duality and $E$ becomes a Banach module with respect to convolution over $L^1_{\omega_E}$ and a Banach module with respect to multiplication over $\mathcal{F}L^1_{\nu_E}$. Because of this, both the TMIB and the DTMIB spaces can be considered as Banach spaces of distributions with two module structures in the sense of \cite{brfe83,fei84} (see also \cite{fegu90}).

\begin{remark}\label{rem-f-con-map}
Let $E$ be a TMIB space and denote $\check{E}'=\{f\in\SSS'(\RR^d)|\, \check{f}\in E'\}$ with norm $\|f\|_{\check{E}'}=\|\check{f}\|_{E'}$. The convolution mapping $*:\check{E}'\times \SSS(\RR^d)\rightarrow \SSS'(\RR^d)$ uniquely extends to a continuous bilinear mapping $*:\check{E}'\times E\rightarrow L^{\infty}_{1/\check{\omega}_E}(\RR^d)$, $f*e(x)=\langle \check{f},T_{-x}e\rangle$, and
\beqs
\|f*e\|_{L^{\infty}_{1/\check{\omega}_E}}= \|(f*e)/\check{\omega}_E\|_{L^{\infty}}\leq \|\check{f}\|_{E'}\|e\|_E,\quad f\in\check{E}',\, e\in E;
\eeqs
furthermore, $f*e\in\mathcal{C}(\RR^d)$, for all $f\in\check{E}'$, $e\in E$ (see \cite[Proposition 6]{dpv} and the comments before it).
\end{remark}

If $\eta$ is a polynomially bounded weight, the spaces $L^p_{\eta}$, $1\leq p<\infty$, and $\mathcal{C}_{\eta,0}$, as well as their associated Fourier spaces $\mathcal{F}L^p_{\eta}$, $1\leq p<\infty$, and $\mathcal{F}\mathcal{C}_{\eta,0}$ are examples of TMIB spaces, while $L^{\infty}_{\eta}$ is a DTMIB space; when $p=2$ and $\eta(x)=(1+|x|)^s$, $s\in\RR$, $\mathcal{F}L^2_{\eta}$ is just the classical Sobolev space $\mathcal{H}_s$. The Shubin-Sobolev spaces \cite{Shubin} $\mathcal{Q}_s$, $s\in\RR$, are also TMIB spaces. Another example of TMIB spaces are the mixed-norm Lebesgue spaces $L^{p,q}_{\eta}(\RR^{2d})$ $1\leq p,q<\infty$, with $\eta$ a polynomially bounded weight on $\RR^{2d}$; of course $L^{\infty,\infty}_{\eta}(\RR^{2d})=L^{\infty}_{\eta}(\RR^{2d})$ is a DTMIB space.\\
\indent Let $X$ be a TMIB or a DTMIB space on $\RR^{2d}$ and let $g\in\SSS(\RR^d)\backslash\{0\}$. The modulation spaces $\Modsq{X}$ associated to $X$ is defined as \cite{dppv-3}:
\beqs
\Modsq{X}=\{f\in\SSS'(\RR^d)|\, V_gf\in X\},\quad \mbox{with norm}\quad \|f\|_{\Modsq{X}}=\|V_gf\|_X.
\eeqs
In \cite{dppv-3}, it is denoted by $\mathcal{M}^X$; we employ the notation $\Modsq{X}$ instead for readability purposes as often $X$ will be a complicated space. We recall its basic properties and refer to \cite[Section 4]{dppv-3} for the complete account. Different choices of $g$ result in the same space with equivalent norms. When $X$ is a TMIB space so is $\Modsq{X}$, while if $X$ is a DTMIB space then $\Modsq{X}$ is also a DTMIB space. If $\eta$ is a polynomially bounded weight on $\RR^{2d}$, $\Modsq{L^{p,q}_{\eta}}$, $1\leq p,q\leq \infty$, is just the classical modulation space $M^{p,q}_{\eta}$; in this case, we will employ the latter notation.

\subsection{Tensor products of TMIB spaces. Modulation spaces associated to tensor products of TMIB spaces}

Given any two TMIB spaces $E$ and $F$ on $\RR^d$, \cite[Theorem 3.6]{dppv-3} yields that their completed $\epsilon$-tensor product $E\hat{\otimes}_{\epsilon} F$ is a TMIB space on $\RR^{2d}$. The same result also yields that the completed $\pi$-tensor product $E\hat{\otimes}_{\pi} F$ is also a TMIB space whenever either $E$ or $F$ additionally satisfies the weak approximation property (we refer to \cite[Chapter 8 Section 43]{kothe2} for the latter notion). We claim that every TMIB space satisfies the weak approximation property; in fact the following result shows that it always satisfies a stronger property, namely the weak sequential approximation property (see \cite[Section 1]{komatsu3} for the latter notion). In the following result, for $X$ and $Y$ two locally convex spaces, we denote by $\mathcal{L}_c(X,Y)$ the space $\mathcal{L}(X,Y)$ equipped with the topology of uniform convergence on all compact convex circled subsets of $X$; when $X$ is complete it is the same as the topology of uniform convergence on all precompact (or, equivalently, relatively compact) subsets of $X$.

\begin{lemma}\label{wap}
Let $E$ be a TMIB space. Then there exists $S_{n,m}\in E'\otimes E$, $n,m\in\ZZ_+$, such that
\beq
\lim_{m\rightarrow\infty}(\lim_{n\rightarrow \infty}S_{n,m})= \operatorname{Id}\,\,\, \mbox{in}\,\,\, \mathcal{L}_c(E,E).
\eeq
In particular, $E$ satisfies the weak sequential approximation property.
\end{lemma}

\begin{proof} Let $\chi,\varphi\in\SSS(\RR^d)$ be such that $\varphi(0)=1$ and $\int_{\RR^d}\chi(x)dx=1$. For each $n\in \ZZ_+$, set $\varphi_n(x)=\varphi(x/n)$ and $\chi_n(x)=n^d\chi(nx)$. Define $G_n:E\rightarrow \SSS(\RR^d)$, $G_n(f)=\chi_n*(\varphi_nf)$; of course they are well defined and continuous operators. By \cite[Corollary 3.3]{dppv-3}, $G_n(f)\rightarrow f$ in $E$, for all $f\in E$, and thus the Banach-Steinhaus theorem implies $G_n\rightarrow\operatorname{Id}$ in $\mathcal{L}_c(E,E)$.\\
\indent Take a Schauder basis $\varphi_j$, $j\in\ZZ_+$, for $\SSS(\RR^d)$; for example the Hermite functions. Then
\beqs
\varphi=\lim_{n\rightarrow \infty}\sum_{j=1}^n \langle f_j,\varphi\rangle\varphi_j\,\, \mbox{in}\,\, \SSS(\RR^d),\quad \forall\varphi\in\SSS(\RR^d),
\eeqs
where $f_j\in\SSS'(\RR^d)$, $j\in\ZZ_+$, are the coefficient functionals. Set
\beqs
\tilde{S}_n=\sum_{j=1}^n f_j\otimes \varphi_j\in\SSS'(\RR^d)\otimes \SSS(\RR^d),\quad n\in\ZZ_+.
\eeqs
Then the Banach-Steinhaus theorem \cite[Theorem 4.5, p. 85]{Sch} implies $\tilde{S}_n\rightarrow \operatorname{Id}$ in $\mathcal{L}_c(\SSS(\RR^d),\SSS(\RR^d))=\mathcal{L}_b(\SSS(\RR^d),\SSS(\RR^d))$ (as $\SSS(\RR^d)$ is a Montel space). For $n,m\in\ZZ_+$, we define
\beq
S_{n,m}=\tilde{S}_n\circ G_m=\sum_{j=1}^n\varphi_m(\check{\chi}_m*f_j)\otimes \varphi_j\in \SSS(\RR^d)\otimes \SSS(\RR^d)\subseteq E'\otimes E.
\eeq
Now, for each fixed $m\in\ZZ_+$, $S_{n,m}\rightarrow G_m$, as $n\rightarrow\infty$, in $\mathcal{L}_c(E,\SSS(\RR^d))$ and consequently in $\mathcal{L}_c(E,E)$ as well. As we already proved that $G_m\rightarrow \operatorname{Id}$ in $\mathcal{L}_c(E,E)$, the proof is complete.
\end{proof}

\begin{remark}
The same proof works in the case when $E$ is a TMIB space of ultradistributions of class $(M_p)$ or $\{M_p\}$ in the sense of \cite{dppv-3}, where the sequence $M_p$, $p\in\NN$, satisfies the same conditions as in \cite{dppv-3} (notice that we never used functions with compact support in the proof!); we point out that when $M_p$, $p\in\NN$, satisfies the same conditions as in \cite{dppv-3}, the Hermite functions are Schauder bases for the Gelfand-Shilov spaces $\SSS^{(M_p)}(\RR^d)$ and $\SSS^{\{M_p\}}(\RR^d)$ as well (see \cite{lan,dj-j}). Consequently, the modulation spaces associated to TMIB spaces of ultradistributions of class $(M_p)$ or $\{M_p\}$ \cite[Section 4]{dppv-3} (see also \cite{C-P-R-T,tof-na}, \cite[Section 5]{cor-nic-rod}) also satisfy the weak sequential approximation property.
\end{remark}

\begin{remark}
In \cite{drw}, it is shown that the mixed-norm Lebesgue spaces $L^{p_1,p_2}_{\eta}$, $p_1,p_2\in[1,\infty)$, and the classical modulation spaces $M^{p_1,p_2}_{\eta}$, $p_1,p_2\in[1,\infty)$, satisfy the so called metric approximation property when $\eta$ satisfies certain conditions. This amounts to saying that the identity mapping can be approximated (in the topology of compact convex circled convergence) with a net consisting of finite rank operators having operator norm bounded by $1$. Lemma \ref{wap} supplements these results as it claims that the approximation can, in fact, be done by a sequence of sequences for a far more general class of spaces; but we do not claim that $S_{n,m}\in E'\otimes E$, $n,m\in\ZZ_+$, is uniformly bounded in the operator topology.
\end{remark}

Consequently, if $E$ and $F$ are any two TMIB spaces on $\RR^d$, $E\hat{\otimes}_{\pi} F$ is always a TMIB space on $\RR^{2d}$; furthermore \cite[Theorem 12, p. 240]{kothe2} implies that $E\hat{\otimes}_{\pi} F$ is continuously included into $E\hat{\otimes}_{\epsilon} F$. Thus, if $E_j$ and $F_j$, $j=1,2$, are four TMIB spaces on $\RR^d$ such that $E_1\hookrightarrow E_2$ and $F_1\hookrightarrow F_2$, the facts that $E_j\hat{\otimes}_{\pi} F_j$ and $E_j\hat{\otimes}_{\epsilon} F_j$, $j=1,2$, are TMIB spaces on $\RR^{2d}$ imply that the following diagram commutes
\begin{center}
\begin{tikzpicture}
  \matrix (m) [matrix of math nodes,row sep={1.3em}, column sep={3em}]
  {
  & E_1\hat{\otimes}_{\pi} F_1 & E_1\hat{\otimes}_{\epsilon} F_1 &\\
  \SSS(\RR^{2d}) & & & \SSS'(\RR^{2d})\\
  & E_2\hat{\otimes}_{\pi} F_2 & E_2\hat{\otimes}_{\epsilon} F_2 &\\
   };
   \draw[right hook ->] (m-2-1) -- (m-1-2);
   \draw[right hook ->] (m-2-1) -- (m-3-2);
   \draw[right hook ->] (m-1-2) -- (m-1-3);
   \draw[right hook ->] (m-1-2) -- (m-3-2);
   \draw[right hook ->] (m-3-2) -- (m-3-3);
   \draw[right hook ->] (m-1-3) -- (m-2-4);
   \draw[right hook ->] (m-1-3) -- (m-3-3);
   \draw[right hook ->] (m-3-3) -- (m-2-4);
\end{tikzpicture}
\end{center}

\section{Amalgam spaces with TMIB and DTMIB spaces as local components}\label{amalgam-spaces-s}

The theory of amalgam spaces was developed more than 40 years ago for a more general class of spaces than the (D)TMIB space \cite{feich,fe81} (see also \cite{ben-oh,fe87-1,he03}). However, keeping our end goal in mind, we defined them only for (D)TMIB spaces and recall only the properties we need.\\
\indent We start by recalling the definition of a bounded uniform partition of unity \cite[Definition]{feich}. For any measurable set $A\subseteq \RR^d$, denote by $\theta_A$ the characteristic function of $A$. Let $E$ be a TMIB or a DTMIB space of distributions on $\RR^d$. Let $\Lambda$ be a countable index set, $U\subseteq \RR^d$ a bounded open neighbourhood of $0$ and $M>0$. A family $\{\phi_{\lambda}\in\DD(\RR^d)|\, \lambda\in\Lambda\}$ is called a \textit{bounded uniform partition of unity} of norm $M$ and size $U$ for $E$ (from now, abbreviated as BUPU), if there exists a discrete family of points $\{y_{\lambda}\}_{\lambda\in\Lambda}$ such that:
\begin{itemize}
\item[$(i)$] $\|\phi_{\lambda}\|_{\mathcal{F}L^1_{\nu_E}}\leq M$, for all $\lambda\in\Lambda$;
\item[$(ii)$] $\supp \phi_{\lambda}\subseteq y_{\lambda}+U$, for all $\lambda\in\Lambda$;
\item[$(iii)$] for any compact subset $K$ of $\RR^d$ there exists $C_K>0$ such that $|\{\lambda\in\Lambda|\, x\in y_{\lambda}+K\}|\leq C_K$, for all $x\in\RR^d$;
\item[$(iv)$] $\sum_{\lambda\in\Lambda} \phi_{\lambda}(x)=1$, for all $x\in\RR^d$.\footnote{the series is locally finite because of $(ii)$ and $(iii)$}
\end{itemize}
Let $\eta$ be a polynomially bounded weight on $\RR^d$. Given any such BUPU $\{\phi_{\lambda}\}_{\lambda\in\Lambda}$ and a compact set $W\supseteq U$, one defines the amalgam space $W(E,L^p_{\eta})$, $1\leq p\leq \infty$, with local component $E$ and global component $L^p_{\eta}$, as the space of all $f\in E_{\operatorname{loc}}=\{f\in\DD'(\RR^d)|\, \chi f\in E,\, \forall \chi\in\DD(\RR^d)\}$ such that
\beq\label{nor-am-s}
\left\|\sum_{\lambda\in \Lambda}\|f\phi_{\lambda}\|_E T_{y_{\lambda}}\theta_W\right\|_{L^p_{\eta}(\RR^d)}<\infty;
\eeq
$W(E,L^{\infty}_{\eta,0})$ is defined analogously: it is the space of all $f\in E_{\operatorname{loc}}$ such that $\sum_{\lambda\in \Lambda}\|f\phi_{\lambda}\|_E T_{y_{\lambda}}\theta_W\in L^{\infty}_{\eta,0}(\RR^d)$ with norm \eqref{nor-am-s} with $p=\infty$. Employing similar technique as in the proof of \cite[Theorem 2]{feich}, one can prove that any other BUPU $\{\tilde{\phi}_{\tilde{\lambda}}\}_{\tilde{\lambda}\in\tilde{\Lambda}}$ with size $\tilde{U}$ and norm $\tilde{M}$ and $\tilde{W}\supseteq \tilde{U}$ any compact set gives the same spaces $W(E,L^p_{\eta})$, $1\leq p\leq\infty$, and $W(E,L^{\infty}_{\eta,0})$ with equivalent norms.\\
\indent Let $\phi\in \DD(\RR^d)$, be such that $0\leq \phi\leq 1$, $\supp\phi\subseteq (-1,1)^d$ and
\beq
\sum_{\mathbf{k}\in\ZZ^d}\phi_{\mathbf{k}}(x)=1,\,\,\, \forall x\in\RR^d,\quad \mbox{where}\,\,\, \phi_{\mathbf{k}}=T_{\mathbf{k}}\phi,\, \mathbf{k}\in\ZZ^d.
\eeq
Then $\{\phi_{\mathbf{k}}\}_{\mathbf{k}\in\ZZ^d}$ is a BUPU of size $(-1,1)^d$ for any TMIB or DTMIB space $E$ with $M=\|\phi\|_{\mathcal{F}L^1_{\nu_E}}(=\|\phi_{\mathbf{k}}\|_{\mathcal{F}L^1_{\nu_E}},\, \forall \mathbf{k}\in\ZZ^d)$ and $y_{\mathbf{k}}=\mathbf{k}$, $\mathbf{k}\in\ZZ^d$. From now on we will always employ this BUPU and $W=[-1,1]^d$ to define the norms on $W(E,L^p_{\eta})$, $1\leq p\leq\infty$, and $W(E,L^{\infty}_{\eta,0})$, for any TMIB or DTMIB space $E$. It is straightforward to check that $W(E,L^p_{\eta})$, $1\leq p\leq\infty$, and $W(E,L^{\infty}_{\eta,0})$ are Banach spaces.

\begin{remark}\label{rem-eq-nor-se-iis}
Let $E$ be a TMIB or a DTMIB space on $\RR^d$ and $\eta$ a polynomially bounded weight on $\RR^d$. It is straightforward to verify that for any $p\in[1,\infty]$ and $f\in E_{\operatorname{loc}}$, $f\in W(E,L^p_{\eta})$ if and only if $\{\|f\phi_{\mathbf{k}}\|_E\eta(\mathbf{k})\}_{\mathbf{k}\in\ZZ^d}\in\ell^p(\ZZ^d)$ and there is $C\geq 1$ such that
\beqs
C^{-1}\|\{\|f\phi_{\mathbf{k}}\|_E \eta(\mathbf{k})\}_{\mathbf{k}\in\ZZ^d}\|_{\ell^p(\ZZ^d)}\leq \|f\|_{W(E,L^p_{\eta})}\leq C\|\{\|f\phi_{\mathbf{k}}\|_E\eta(\mathbf{k})\}_{\mathbf{k}\in\ZZ^d}\|_{\ell^p(\ZZ^d)},
\eeqs
for all $f\in W(E,L^p_{\eta})$ and all $p\in[1,\infty]$ (i.e. $C$ is independent of $p$ as well). Similarly, given $f\in E_{\operatorname{loc}}$, $f\in W(E,L^{\infty}_{\eta,0})$ if and only if $\{\|f\phi_{\mathbf{k}}\|_E\eta(\mathbf{k})\}_{\mathbf{k}\in\ZZ^d}\in c_0(\ZZ^d)$ and
\beqs
C^{-1}\|\{\|f\phi_{\mathbf{k}}\|_E \eta(\mathbf{k})\}_{\mathbf{k}\in\ZZ^d}\|_{\ell^{\infty}(\ZZ^d)}\leq \|f\|_{W(E,L^{\infty}_{\eta,0})}\leq C\|\{\|f\phi_{\mathbf{k}}\|_E \eta(\mathbf{k})\}_{\mathbf{k}\in\ZZ^d}\|_{\ell^{\infty}(\ZZ^d)},
\eeqs
for all $f\in W(E,L^{\infty}_{\eta,0})$ with the same constant $C$.
\end{remark}

We collect the properties we need of $W(E,L^p_{\eta})$, $1\leq p\leq\infty$, and $W(E,L^{\infty}_{\eta,0})$ in the following result; here and throughout the rest of the article, we employ the notation $J_n=\ZZ^d\cap[-n,n]^d$, $n\in\ZZ_+$.

\begin{lemma}\label{dual-ams-s1}
Let $\eta$ be a polynomially bounded weight on $\RR^d$.
\begin{itemize}
\item[$(i)$] If $E$ is a TMIB or a DTMIB space on $\RR^d$, then for any $1\leq p_1\leq p_2<\infty$ the following continuous inclusions hold true:
\beq\label{inc-s-fr}
\SSS(\RR^d)\rightarrow  W(E,L^{p_1}_{\eta})\rightarrow W(E,L^{p_2}_{\eta}) \rightarrow W(E,L^{\infty}_{\eta,0})\rightarrow W(E,L^{\infty}_{\eta}) \rightarrow \SSS'(\RR^d).
\eeq
\item[$(ii)$] If $E$ is a TMIB space on $\RR^d$ then $W(E,L^p_{\eta})$, $1\leq p<\infty$, and $W(E,L^{\infty}_{\eta,0})$ are TMIB spaces on $\RR^d$ as well.
\item[$(iii)$] Let $E$ be a TMIB space on $\RR^d$. Then the strong dual $W(E,L^p_{\eta})'_b$ of $W(E,L^p_{\eta})$, $1\leq p<\infty$, is equal to $W(E',L^q_{1/\eta})$, $p^{-1}+q^{-1}=1$, with equivalent norms. Furthermore, the strong dual $W(E,L^{\infty}_{\eta,0})'_b$ of $W(E,L^{\infty}_{\eta,0})$ is equal to $W(E',L^1_{1/\eta})$ with equivalent norms. Consequently, $W(E',L^p_{1/\eta})$, $1\leq p\leq \infty$, are DTMIB spaces on $\RR^d$.
\end{itemize}
\end{lemma}

\begin{proof} With $E$ a TMIB or a DTMIB space on $\RR^d$, it is straightforward to verify that $\SSS(\RR^d)\subseteq W(E,L^1_{\eta})$ and $W(E,L^{\infty}_{\eta})\subseteq \SSS'(\RR^d)$ and that these inclusions are continuous. The rest of part $(i)$ follows from Remark \ref{rem-eq-nor-se-iis} and the continuous inclusions $\ell^{p_1}\subseteq \ell^{p_2}\subseteq c_0\subseteq \ell^{\infty}$, when $1\leq p_1\leq p_2<\infty$.\\
\indent The only non trivial part of $(ii)$ is the density of $\SSS(\RR^d)$ in $W(E,L^p_{\eta})$, $1\leq p<\infty$, and in $W(E,L^{\infty}_{\eta,0})$. We prove this for $W(E,L^p_{\eta})$, $1\leq p<\infty$, as the proof for $W(E,L^{\infty}_{\eta,0})$ is analogous.\\
\indent Pick $\psi\in\DD(\RR^d)$ such that $0\leq \psi\leq 1$, $\psi=1$ on $[-1,1]^d$ and $\supp\psi\subseteq (-2,2)^d$. Denote, $\psi_n(x)=\psi(x/n)$, $n\in\ZZ_+$. Take nonnegative $\chi\in\DD((-1/2,1/2)^d)$ such that $\int_{\RR^d}\chi(x)dx=1$ and set $\chi_n(x)=n^d\chi(nx)$, $n\in\ZZ_+$; consequently $\int_{\RR^d}\chi_n(x)dx=1$, $\forall n\in\ZZ_+$. Let $f\in W(E,L^p_{\eta})$ and $\varepsilon>0$ be arbitrary but fixed. In view of Remark \ref{rem-eq-nor-se-iis}, there exists $n_1\in\ZZ_+$ such that
\beqs
\left(\sum_{\mathbf{k}\in\ZZ^d\backslash J_{n_1}}\|\phi_{\mathbf{k}}f\|_E^p\eta(\mathbf{k})^p\right)^{1/p}\leq \varepsilon/(4C(1+\sup_{n\in\ZZ_+}\|\psi_n\|_{\mathcal{F}L^1_{\nu_E}})).
\eeqs
with $C\geq 1$ being the constant from Remark \ref{rem-eq-nor-se-iis} (it is straightforward to check $\sup_{n\in\ZZ_+}\|\psi_n\|_{\mathcal{F}L^1_{\nu_E}}<\infty$). Furthermore, \cite[Corollary 1]{dpv} implies that for each $e\in E$, $\chi_n*e\rightarrow e$, as $n\rightarrow \infty$, in $E$. Consequently, there exists $n_2\in\ZZ_+$ such that
\beq\label{est-ter-sed}
\|\psi_{n_1+1}f-\chi_{n_2}*(\psi_{n_1+1}f)\|_E\leq \varepsilon\cdot \left(2C\|\phi\|_{\mathcal{F}L^1_{\nu_E}}(4n_1+7)^d \cdot \max_{\mathbf{k}\in J_{2n_1+3}} \eta(\mathbf{k})\right)^{-1}.
\eeq
Clearly, $\chi_{n_2}*(\psi_{n_1+1}f)\in \DD((-2n_1-3,2n_1+3)^d)$. Notice that
\begin{align*}
\|f-&\chi_{n_2}*(\psi_{n_1+1} f)\|_{W(E,L^p_{\eta})}\\
&\leq C\left(\sum_{\mathbf{k}\in\ZZ^d} \|\phi_{\mathbf{k}} f-\phi_{\mathbf{k}}\psi_{n_1+1} f\|_E^p\eta(\mathbf{k})^p\right)^{1/p}\\
&\quad+ C\left(\sum_{\mathbf{k}\in\ZZ^d} \|\phi_{\mathbf{k}}\psi_{n_1+1} f- \phi_{\mathbf{k}}(\chi_{n_2}*(\psi_{n_1+1} f))\|_E^p\eta(\mathbf{k})^p\right)^{1/p}=S_1+S_2.
\end{align*}
To estimate $S_1$, notice that $\psi_{n_1+1}=1$ on $[-n_1-1,n_1+1]^d$ and thus
\beqs
S_1\leq C\left(\sum_{\mathbf{k}\in\ZZ^d\backslash J_{n_1}} \|\phi_{\mathbf{k}} f\|_E^p\eta(\mathbf{k})^p\right)^{1/p}+C\|\psi_{n_1+1}\|_{\mathcal{F}L^1_{\nu_E}} \left(\sum_{\mathbf{k}\in\ZZ^d\backslash J_{n_1}} \|\phi_{\mathbf{k}} f\|_E^p\eta(\mathbf{k})^p\right)^{1/p}\leq \frac{\varepsilon}{2}.
\eeqs
To estimate $S_2$, observe that all of the terms where $\mathbf{k}\not\in J_{2n_1+3}$ are zero, so \eqref{est-ter-sed} implies
\beqs
S_2\leq C\|\phi\|_{\mathcal{F}L^1_{\nu_E}}\left(\sum_{\mathbf{k}\in J_{2n_1+3}} \|\psi_{n_1+1} f-\chi_{n_2}*(\psi_{n_1+1} f)\|_E^p\eta(\mathbf{k})^p\right)^{1/p}\leq \varepsilon/2.
\eeqs
Consequently, $\|f-\chi_{n_2}*(\psi_{n_1+1} f)\|_{W(E,L^p_{\eta})}\leq \varepsilon$ which completes the proof of $(ii)$.\\
\indent Now, in view of Remark \ref{rem-eq-nor-se-iis}, part $(iii)$ follows from \cite[Theorem 2.8]{fegr85}.
\end{proof}

\begin{remark}
If $E$ is a TMIB space on $\RR^d$, Remark \ref{rem-eq-nor-se-iis} together with Lemma \ref{dual-ams-s1} yield
\beq\label{inc-am-sp-n}
\SSS(\RR^d)\hookrightarrow W(E,L^1)\hookrightarrow E\hookrightarrow W(E,L^{\infty}_0) \hookrightarrow \SSS'(\RR^d).
\eeq
Consequently, Lemma \ref{dual-ams-s1} $(iii)$ gives the following continuous inclusions:
\beq
\SSS(\RR^d)\rightarrow W(E',L^1)\rightarrow E'\rightarrow W(E',L^{\infty}) \rightarrow \SSS'(\RR^d).
\eeq
\end{remark}

Let $E$ be a TMIB or a DTMIB space on $\RR^d$. Since for any $g\in \mathcal{F}L^1_{\nu_E}$, the mapping $g\mapsto T_xg$, $\RR^d\rightarrow \mathcal{F}L^1_{\nu_E}$, is continuous, the mapping $x\mapsto eT_x g$, $\RR^d\rightarrow E$, is continuous for all $e\in E$, $g\in\mathcal{F}L^1_{\nu_E}$. Thus given any $f\in E_{\operatorname{loc}}$ and $\chi\in\DD(\RR^d)\backslash\{0\}$, $x\rightarrow \|fT_x\chi\|_E$, $\RR^d\rightarrow [0,\infty)$, is a continuous function and therefore it make sense to take its $L^p_{\eta}$ norm for any $1\leq p\leq\infty$; of course, it may be infinite. The elements of the amalgam space $W(E,L^p_{\eta})$ can be characterised by the finiteness of this norm.

\begin{lemma}\label{equ-con-cdis-nor}
Let $E$ be a TMIB or a DTMIB space on $\RR^d$, $\eta$ a polynomially bounded weight on $\RR^d$ and $\chi\in\DD(\RR^d)\backslash\{0\}$. For any $1\leq p\leq\infty$ and $f\in E_{\operatorname{loc}}$, $f\in W(E,L^p_{\eta})$ if and only if $x\mapsto \|fT_x\chi\|_E$, $\RR^d\rightarrow [0,\infty)$, belongs to $L^p_{\eta}(\RR^d)$. Furthermore,
\beq\label{nor-int-a}
f\mapsto \left(\int_{\RR^d}\|fT_x\chi\|_E^p\eta(x)^pdx\right)^{1/p}
\eeq
is a norm on $W(E,L^p_{\eta})$ equivalent to $\|\cdot\|_{W(E,L^p_{\eta})}$.\\
\indent For any $f\in E_{\operatorname{loc}}$, $f\in W(E,L^{\infty}_{\eta,0})$ if and only if $x\mapsto \|fT_x\chi\|_E$, $\RR^d\rightarrow [0,\infty)$, belongs to $\mathcal{C}_{\eta,0}(\RR^d)$ and $\|\cdot\|_{W(E,L^{\infty}_{\eta,0})}$ is equivalent to the norm \eqref{nor-int-a} with $p=\infty$.
\end{lemma}

\begin{proof} Let $1\leq p\leq\infty$. Denote by $X_p$ the subspace of $E_{\operatorname{loc}}$ consisting of those elements for which \eqref{nor-int-a} is finite; we denote this quantity by $\|\cdot\|_{X_p}$. Clearly, it is a seminorm on $X_p$ (the fact that it is a norm will follow from the second part of the proof). Without losing generality, we can assume $\chi(0)\neq 0$; otherwise we can always change variables in \eqref{nor-int-a} without changing $X_p$ and the new seminorm is equivalent to the old one (because $\eta$ is a polynomially bounded weight). There is $k_0\in \ZZ_+$ such that $\sum_{\mathbf{r}\in J_{k_0}}\phi_{\mathbf{r}}=1$ on $\supp\chi$. Let $f\in W(E,L^p_{\eta})$. We estimate as follows
\beqs
\|f\|_{X_p}&\leq& \|\chi\|_{\mathcal{F}L^1_{\nu_E}} \left(\int_{\RR^d}\left(\sum_{\mathbf{r}\in J_{k_0}}\|fT_{x+\mathbf{r}}\phi\|_E\right)^p\eta(x)^pdx\right)^{1/p}\\
&\leq& C_1 \sum_{\mathbf{r}\in J_{k_0}} \left(\int_{\RR^d}\|fT_{x+\mathbf{r}}\phi\|_E^p\eta(x)^pdx\right)^{1/p} \leq C_2 \left(\int_{\RR^d}\|fT_x\phi\|_E^p\eta(x)^pdx\right)^{1/p}\\
&\leq& C_2 \left(\sum_{\mathbf{k}\in\ZZ^d}\int_{\mathbf{k}+[-1,1]^d} \|fT_x\phi\|_E^p\eta(x)^pdx\right)^{1/p}.
\eeqs
Let $\mathbf{k}\in\ZZ^d$ and $x\in \mathbf{k}+[-1,1]^d$ be arbitrary but fixed. Then
\beqs
\|f T_x\phi\|_E \leq \sum_{\mathbf{r}\in J_2}\|f\phi_{\mathbf{k}+\mathbf{r}}T_x\phi\|_E \leq \|\phi\|_{\mathcal{F}L^1_{\nu_E}}\sum_{\mathbf{r}\in J_2} \|\phi_{\mathbf{k}+\mathbf{r}}f\|_E.
\eeqs
Consequently,
\beqs
\|f\|_{X_p}\leq C_3\sum_{\mathbf{r}\in J_2} \left(\sum_{\mathbf{k}\in\ZZ^d} \int_{\mathbf{k}+[-1,1]^d} \|\phi_{\mathbf{k}+\mathbf{r}}f\|_E^p \eta(x)^pdx\right)^{1/p}.
\eeqs
We introduce the change of variables $t=x+\mathbf{r}$ in each of the integrals. As $\eta$ is a polynomially bounded weight, we deduce
\beqs
\|f\|_{X_p}\leq C_4 \left(\sum_{\mathbf{k}\in\ZZ^d} \int_{\mathbf{k}+[-1,1]^d} \|\phi_{\mathbf{k}}f\|_E^p \eta(t)^p dt\right)^{1/p}\leq C_5\|f\|_{W(E, L^p_{\eta})}.
\eeqs
We turn our attention to the opposite inclusion. As $\chi(0)\neq 0$, there exists $n_0\in\ZZ_+$ such that $\chi\neq0$ on $[-1/n_0,1/n_0]^d$. Pick $\psi\in\DD(\RR^d)$ such that $\psi\chi=1$ on $[-1/n_0,1/n_0]^d$. Define $\tilde{\phi}(x)=\phi(n_0x)$, $x\in\RR^d$; clearly
\beqs
\supp\tilde{\phi}\subseteq (-1/n_0,1/n_0)^d\quad \mbox{and}\quad \sum_{\mathbf{k}\in\ZZ^d} T_{\mathbf{k}/n_0}\tilde{\phi}(x)=1,\,\, \forall x\in\RR^d.
\eeqs
Let $f\in X_p$. For $\mathbf{k}\in\ZZ^d$ and $t\in\mathbf{k}+[-1,1]^d$, we infer
\beqs
\|f\phi_{\mathbf{k}}\|_E\leq \sum_{\mathbf{r}\in J_{2n_0+1}}\|f\phi_{\mathbf{k}}T_{t+\mathbf{r}/n_0}\tilde{\phi}\|_E\leq \|\phi\|_{\mathcal{F}L^1_{\nu_E}}\sum_{\mathbf{r}\in J_{2n_0+1}}\|fT_{t+\mathbf{r}/n_0}\tilde{\phi}\|_E.
\eeqs
Consequently,
\beqs
\|f\|_{W(E,L^p_{\eta})}\leq C'_1\sum_{\mathbf{r}\in J_{2n_0+1}} \left(\sum_{\mathbf{k}\in\ZZ^d}\int_{\mathbf{k}+[-1,1]^d} \|fT_{t+\mathbf{r}/n_0}\tilde{\phi}\|_E^p\eta(t)^pdt\right)^{1/p}.
\eeqs
Introducing the change of variables $x=t+\mathbf{r}/n_0$ together with the fact that $\eta$ is polynomially bounded weight, we conclude
\beqs
\|f\|_{W(E,L^p_{\eta})}\leq C'_2 \left(\sum_{\mathbf{k}\in\ZZ^d}\int_{\mathbf{k}+[-4,4]^d} \|fT_x\tilde{\phi}\|_E^p\eta(x)^pdx\right)^{1/p}\leq C'_3 \left(\int_{\RR^d} \|fT_x\tilde{\phi}\|_E^p\eta(x)^pdx\right)^{1/p}.
\eeqs
Since $\tilde{\phi}=\tilde{\phi}\psi\chi$, we immediately deduce $\|f\|_{W(E,L^p_{\eta})}\leq C'_3\|\tilde{\phi}\psi\|_{\mathcal{F}L^1_{\nu_E}}\|f\|_{X_p}$. This implies that $\|\cdot\|_{X_p}$ is in fact a norm on $X_p$, and in view of the first part, $X_p=W(E,L^p_{\eta})$ with $\|\cdot\|_{X_p}$ being equivalent to $\|\cdot\|_{W(E,L^p_{\eta})}$.\\
\indent The proof for $W(E,L^{\infty}_{\eta,0})$ is analogous and we omit it; since $x\mapsto \|fT_x\chi\|_E$ is continuous, this function belongs to $L^{\infty}_{\eta,0}(\RR^d)$ if and only if it belongs to $\mathcal{C}_{\eta,0}(\RR^d)$.
\end{proof}

We end this section with the following result which connects the classical modulation spaces with the amalgam spaces having $\mathcal{F}L^p_{\eta}$ as a global component. This result has appeared in other works in the past (cf. \cite{cor-nic,CR,fei90}); we provide a proof for the sake of completeness.

\begin{lemma}\label{ident-mod-amalg}
Let $\eta_1$ and $\eta_2$ be two polynomially bounded weights and $p_1,p_2\in[1,\infty)$. Then $M^{p_1,p_2}_{\eta_1\otimes\eta_2}=\mathcal{F}^{-1}W(\mathcal{F}L^{p_1}_{\eta_1}, L^{p_2}_{\eta_2})$ with equivalent norms.
\end{lemma}

\begin{proof} Set $g=\mathcal{F}^{-1}\phi\in\SSS(\RR^d)\backslash\{0\}$. In the proof, we employ the following identity (see \cite[Lemma 3.1.1, p. 39]{Grochenig}): for any $f\in\SSS'(\RR^d)$ it holds that
\beq\label{iden-stftb-id}
V_gf(x,\xi)=e^{-2\pi i x\xi}\mathcal{F}(\mathcal{F}f\cdot T_{\xi}\phi)(-x),\quad \mbox{for all}\,\, x,\xi\in\RR^d.
\eeq
Let $f\in \mathcal{F}^{-1}W(\mathcal{F}L^{p_1}_{\eta_1}, L^{p_2}_{\eta_2})$. Employing \eqref{iden-stftb-id} and the fact $\mathcal{F}^{-1}L^{p_1}_{\check{\eta}_1}=\mathcal{F}L^{p_1}_{\eta_1}$ with $\|\cdot\|_{\mathcal{F}^{-1}L^{p_1}_{\check{\eta}_1}}= \|\cdot\|_{\mathcal{F}L^{p_1}_{\eta_1}}$, we infer
\beqs
\|f\|_{M^{p_1,p_2}_{\eta_1\otimes\eta_2}}&=&\left(\int_{\RR^d}\|(\mathcal{F}f) (T_{\xi}\phi)\|^{p_2}_{\mathcal{F}L^{p_1}_{\eta_1}}\eta_2(\xi)^{p_2}d\xi\right)^{1/p_2}\\
&\leq& \left(\sum_{\mathbf{k}\in\ZZ^d} \int_{\mathbf{k}+[-1,1]^d}\|(\mathcal{F}f) (T_{\xi}\phi)\|^{p_2}_{\mathcal{F}L^{p_1}_{\eta_1}}\eta_2(\xi)^{p_2}d\xi\right)^{1/p_2}.
\eeqs
Let $\mathbf{k}\in\ZZ^d$ and $\xi\in \mathbf{k}+[-1,1]^d$ be arbitrary but fixed. Then
\beqs
\|(\mathcal{F}f) (T_{\xi}\phi)\|_{\mathcal{F}L^{p_1}_{\eta_1}} \leq \sum_{\mathbf{r}\in J_2}\|(\mathcal{F}f) (T_{\xi}\phi)\phi_{\mathbf{k}+\mathbf{r}}\|_{\mathcal{F}L^{p_1}_{\eta_1}} \leq C_1\sum_{\mathbf{r}\in J_2} \|\phi_{\mathbf{k}+\mathbf{r}}\mathcal{F}f\|_{\mathcal{F}L^{p_1}_{\eta_1}}.
\eeqs
Consequently,
\beqs
\|f\|_{M^{p_1,p_2}_{\eta_1\otimes\eta_2}}\leq C_2\sum_{\mathbf{r}\in J_2} \left(\sum_{\mathbf{k}\in\ZZ^d} \int_{\mathbf{k}+[-1,1]^d} \|\phi_{\mathbf{k}+\mathbf{r}}\mathcal{F}f\|^{p_2}_{\mathcal{F}L^{p_1}_{\eta_1}} \eta_2(\xi)^{p_2}d\xi\right)^{1/p_2}.
\eeqs
We introduce the change of variables $t=\xi+\mathbf{r}$ in each of the integrals. As $\eta_2$ is a polynomially bounded weight, we deduce
\beqs
\|f\|_{M^{p_1,p_2}_{\eta_1\otimes\eta_2}}\leq C_3 \left(\sum_{\mathbf{k}\in\ZZ^d} \int_{\mathbf{k}+[-1,1]^d} \|\phi_{\mathbf{k}}\mathcal{F}f\|^{p_2}_{\mathcal{F}L^{p_1}_{\eta_1}} \eta_2(t)^{p_2}dt\right)^{1/p_2}\leq C_3\|f\|_{\mathcal{F}^{-1}W(\mathcal{F}L^{p_1}_{\eta_1}, L^{p_2}_{\eta_2})}.
\eeqs
Assume now $f\in M^{p_1,p_2}_{\eta_1\otimes\eta_2}$. For $\mathbf{k}\in\ZZ^d$ and $t\in \mathbf{k}+[-1,1]^d$, we infer
\beqs
\|f*\mathcal{F}^{-1}\phi_{\mathbf{k}}\|_{L^{p_1}_{\eta_1}}\leq \sum_{\mathbf{r}\in J_2} \|f*\mathcal{F}^{-1}(\phi_{\mathbf{k}})* (\mathcal{F}^{-1}T_{t+\mathbf{r}}\phi)\|_{L^{p_1}_{\eta_1}} \leq C'_1 \sum_{\mathbf{r}\in J_2} \|f* \mathcal{F}^{-1}T_{t+\mathbf{r}}\phi\|_{L^{p_1}_{\eta_1}}.
\eeqs
Denoting by $\theta$ the characteristic function of $[-1,1]^d$, we deduce
\beqs
I&=&\left(\int_{\RR^d}\left(\sum_{\mathbf{k}\in\ZZ^d} \|f*\mathcal{F}^{-1}\phi_{\mathbf{k}}\|_{L^{p_1}_{\eta_1}} \theta(t-\mathbf{k})\eta_2(t)\right)^{p_2} dt\right)^{1/p_2}\\
&\leq& C'_2\sum_{\mathbf{r}\in J_2}\left(\sum_{\mathbf{k}\in\ZZ^d} \int_{\mathbf{k}+[-1,1]^d} \|f*\mathcal{F}^{-1}T_{t+\mathbf{r}}\phi\|^{p_2}_{L^{p_1}_{\eta_1}} \eta_2(t)^{p_2} dt\right)^{1/p_2}.
\eeqs
Introducing the change of variables $\xi=t+\mathbf{r}$ together with the fact that $\eta_2$ is polynomially bounded weight, we conclude
\beqs
I&\leq& C'_3\left(\sum_{\mathbf{k}\in\ZZ^d} \int_{\mathbf{k}+[-1,1]^d} \|f*\mathcal{F}^{-1}T_{\xi}\phi\|^{p_2}_{L^{p_1}_{\eta_1}} \eta_2(\xi)^{p_2} d\xi\right)^{1/p_2}\\
&\leq& C'_4\left(\int_{\RR^d} \|f*\mathcal{F}^{-1}T_{\xi}\phi\|^{p_2}_{L^{p_1}_{\eta_1}} \eta_2(\xi)^{p_2} d\xi\right)^{1/p_2}.
\eeqs
Because of \eqref{iden-stftb-id}, the very last term is just $\|f\|_{M^{p_1,p_2}_{\eta_1\otimes\eta_2}}$. Furthermore, the above also implies $\phi_{\mathbf{k}}\mathcal{F}f=\mathcal{F}(f*\mathcal{F}^{-1}\phi_{\mathbf{k}})\in \mathcal{F}L^{p_1}_{\eta_1}$, $\forall \mathbf{k}\in\ZZ^d$, which, in turn, yields that $\mathcal{F}f$ is locally in $\mathcal{F}L^{p_1}_{\eta_1}$ (i.e. in $(\mathcal{F}L^{p_1}_{\eta_1})_{\operatorname{loc}}$). As $I$ is just $\|f\|_{\mathcal{F}^{-1}W(\mathcal{F}L^{p_1}_{\eta_1},L^{p_2}_{\eta_2})}$, the proof is complete.
\end{proof}

The case when $p_1=p_2$ in the above lemma has been discussed in great detail in \cite{fei90}. There it has been shown that, under certain conditions on the weights $\eta_1$ and $\eta_2$, the Fourier image of $W(\mathcal{F}L^p_{\eta_1},L^p_{\eta_2})$ is exactly $W(\mathcal{F}L^p_{\eta_2},L^p_{\eta_1})$; this gives the Fourier invariance of the spaces $M^{p,p}$ (cf. \cite[Theorem 11.3.5, p. 236]{Grochenig}).

\section{Modulation spaces associated to a class of completed $\pi$-tensor products of TMIB spaces}\label{spi}

In this section we identify the modulation spaces associated to a large class of completed $\pi$-tensor products of TMIB spaces. In order to better illustrate the key ideas in the main result, we first consider the case $\Modsq{L^p\hat{\otimes}_{\pi} L^p}$, $p\in[1,2]$; of course the claim in the next lemma is a consequence of the main result: Theorem \ref{1.21} (see Corollary \ref{cor12} below).

\begin{lemma}\label{dipiprxx3}
For every $p\in[1,2]$, $\Modsq{L^p\hat{\otimes}_{\pi} L^p}=W(\mathcal{F}L^p,L^1)$ with equivalent norms.
\end{lemma}

\begin{proof} First we prove the inclusion $\Modsq{L^p\hat{\otimes}_{\pi} L^p}\subseteq W(\mathcal{F}L^p,L^1)$. Every element of $L^p\hat{\otimes}_{\pi} L^p$ can be represented as an absolutely convergent series
\beq\label{rep-f-fun-inssss}
\sum_{j=1}^{\infty} \lambda_j\varphi_j\otimes \psi_j
\eeq
where $\varphi_j,\psi_j\in\SSS(\RR^d)$, $\forall j\in \ZZ_+$, $\{\varphi_j\}_{j\in\ZZ_+}$ and $\{\psi_j\}_{j\in\ZZ_+}$ are bounded subsets of $L^p(\RR^d)$ and $\lambda_j\in\CC$, $j\in\ZZ_+$, is such that $\sum_j|\lambda_j|<\infty$ (see \cite[Theorem 6.4, p. 94]{Sch} and \cite[Theorem 6, p. 188]{kothe2}). Pick $g\in\SSS(\RR^d)\backslash\{0\}$ satisfying $\|g\|_{L^2}=1$. Let $f\in \Modsq{L^p\hat{\otimes}_{\pi} L^p}$ and thus $V_g f\in L^p\hat{\otimes}_{\pi} L^p$. Consequently $V_gf$ has the form \eqref{rep-f-fun-inssss}. The key ingredient is that
\begin{equation} \label{tensact1}
V_g^*(\varphi_j\otimes\psi_j)=(\mathcal{F}^{-1}\psi_j)(\varphi_j*g),\quad \forall j\in\ZZ_+
\end{equation} 
(cf. \cite[Equation (5.1)]{dppv-3}), which, in turn, implies
\beq\label{ser-for-funrr}
f=V_g^*V_gf=\sum_{j=1}^{\infty}\lambda_j V_g^*(\varphi_j\otimes\psi_j)=\sum_{j=1}^{\infty}\lambda_j (\mathcal{F}^{-1}\psi_j)(\varphi_j*g).
\eeq
Now, \cite[Theorem 3]{feich} implies
\beqs
\|\varphi_j*g\|_{W(\mathcal{F}L^1,L^p)}\leq C_1\|\varphi_j\|_{W(L^1,L^p)}\|g\|_{W(\mathcal{F}L^1,L^1)}\leq C'_1 \|\varphi_j\|_{L^p}\|g\|_{W(\mathcal{F}L^1,L^1)},\quad \forall j\in\ZZ_+.
\eeqs
Denoting by $q$ the H\"older conjugate to $p$ and employing Remark \ref{rem-eq-nor-se-iis}, we infer
\beq
\|(\mathcal{F}^{-1}\psi_j)(\varphi_j*g)\|_{W(\mathcal{F}L^p,L^1)}
&\leq& C \sum_{\mathbf{k}\in\ZZ^d} \|\phi_{\mathbf{k}}(\mathcal{F}^{-1}\psi_j)(\varphi_j*g)\|_{\mathcal{F}L^p} \nonumber\\
&\leq& C\sum_{\mathbf{r}\in J_1}\sum_{\mathbf{k}\in\ZZ^d} \|\phi_{\mathbf{k}}(\mathcal{F}^{-1}\psi_j)\|_{\mathcal{F}L^p} \|\phi_{\mathbf{k}+\mathbf{r}}(\varphi_j*g)\|_{\mathcal{F}L^1} \label{est-hh-for-gen-ca}\\
&\leq& C_2 \|\mathcal{F}^{-1}\psi_j\|_{W(\mathcal{F}L^p,L^q)} \|\varphi_j*g\|_{W(\mathcal{F}L^1,L^p)}. \nonumber
\eeq
As $p\in[1,2]$, we have the continuous inclusion $L^p=W(L^p,L^p)\subseteq W(\mathcal{F}L^q,L^p)$ (see \cite[Lemma 1.2 $iv)$]{fe81}) and consequently \cite[Theorem 3.2]{fe81} gives the continuous inclusion $\mathcal{F}L^p\subseteq W(\mathcal{F}L^p,L^q)$. Since $\|\cdot\|_{\mathcal{F}L^p}=\|\cdot\|_{\mathcal{F}^{-1}L^p}$, we deduce
\beqs
\|(\mathcal{F}^{-1}\psi_j)(\varphi_j*g)\|_{W(\mathcal{F}L^p,L^1)}\leq C_3 \|\psi_j\|_{L^p} \|\varphi_j\|_{L^p}\|g\|_{W(\mathcal{F}L^1,L^1)},\quad \forall j\in\ZZ_+.
\eeqs
Consequently, the series in \eqref{ser-for-funrr} is absolutely summable in $W(\mathcal{F}L^p,L^1)$ and thus $f\in W(\mathcal{F}L^p,L^1)$.\\
\indent We turn our attention to the opposite inclusion. Remark \ref{rem-eq-nor-se-iis} implies that $f\in W(\mathcal{F}L^p,L^1)$ is characterised by the finiteness of the expression
\beqs
\sum_{\mathbf{k}\in\ZZ^d}\|f\phi_{\mathbf{k}}\|_{\mathcal{F}L^p}<\infty.
\eeqs
The key idea is that we can split $\phi_{\mathbf{k}}$ into two parts so as $f\phi_{\mathbf{k}}$ is of the same form as the summands in \eqref{ser-for-funrr}. To make this precise, pick $g\in\DD(\RR^d)$ such that $0\leq g\leq 1$ and $g=1$ on $[-2,2]^d$. Then $g*g(x)\geq 1$, $\forall x\in[-1,1]^d$. Take nonnegative $\psi\in\DD(\RR^d)$ such that $\psi=1/(g*g)$ on $[-1,1]^d$. Then
\beqs
f\phi_{\mathbf{k}}=f\phi_{\mathbf{k}}(T_{\mathbf{k}}\psi)((T_{\mathbf{k}}g)*g),\quad \forall \mathbf{k}\in\ZZ^d.
\eeqs
Moreover,
\beqs
\|f\phi_{\mathbf{k}}T_{\mathbf{k}}\psi\|_{\mathcal{F}L^p}\leq \|f\phi_{\mathbf{k}}\|_{\mathcal{F}L^p} \|T_{\mathbf{k}}\psi\|_{\mathcal{F}L^1} =\|f\phi_{\mathbf{k}}\|_{\mathcal{F}L^p} \|\psi\|_{\mathcal{F}L^1},\quad \forall \mathbf{k}\in\ZZ^d.
\eeqs
As $\|\cdot\|_{\mathcal{F}^{-1}L^p}=\|\cdot\|_{\mathcal{F}L^p}$, we deduce $\sum_{\mathbf{k}\in\ZZ^d} (T_{\mathbf{k}}g)\otimes \mathcal{F}(f\phi_{\mathbf{k}}T_{\mathbf{k}}\psi)$ is absolutely summable in $L^p(\RR^d)\hat{\otimes}_{\pi} L^p(\RR^d)$ to some $\tilde{f}\in L^p(\RR^d)\hat{\otimes}_{\pi} L^p(\RR^d)$ and (cf. Remark \ref{rem-eq-nor-se-iis})
\beqs
\|\tilde{f}\|_{L^p(\RR^d)\hat{\otimes}_{\pi} L^p(\RR^d)}\leq C\|\psi\|_{\mathcal{F}L^1}\|g\|_{L^p(\RR^d)}\|f\|_{W(\mathcal{F}L^p,L^1)}.
\eeqs
Since
\beqs
V_g^*\tilde{f}=\sum_{\mathbf{k}\in\ZZ^d} f\phi_{\mathbf{k}}(T_{\mathbf{k}}\psi)((T_{\mathbf{k}}g)*g)=\sum_{\mathbf{k}\in\ZZ^d} f\phi_{\mathbf{k}}=f,
\eeqs
\cite[Proposition 4.4]{dppv-3} verifies that $f\in \Modsq{L^p\hat{\otimes}_{\pi} L^p}$ and the inclusion $W(\mathcal{F}L^p,L^1)\subseteq \Modsq{L^p\hat{\otimes}_{\pi} L^p}$ is continuous. Now the open mapping theorem yields that the norms $\|\cdot\|_{\Modsq{L^p\hat{\otimes}_{\pi} L^p}}$ and $\|\cdot\|_{W(\mathcal{F}L^p,L^1)}$ are equivalent.
\end{proof}

Our goal is to prove a general version of Lemma \ref{dipiprxx3} where each of the spaces $L^p$ is replaced by a TMIB space. The inequality \eqref{est-hh-for-gen-ca} hints that the local and global behaviour of $V_gf$ in each of the variables $x$ and $\xi$ ``come separately''. This suggests to look at tensor products of amalgam spaces as they naturally keep track of the local a global behaviour separately. Another observation from the above proof, albeit not obvious at a glance, is that the local component of $W(\mathcal{F}L^p,L^1)$ came from the second $L^p$ in $L^p\hat{\otimes}_{\pi} L^p$. The significance of the first $L^p$ was in its global growth; aside from the fact $\|T_x\|_{\mathcal{L}_b(L^p(\RR^d))}=1$, the local behaviour did not really mattered. This will become more transparent in the general result (Theorem \ref{1.21}) and its proof.

\subsection{The general case}\label{sub-sec-pi-prod}

The goal of this section is to identify the modulation spaces associated to completed $\pi$-tensor products of a large class of amalgam spaces constructed out of TMIB spaces. The main result is the following.

\begin{theorem}\label{1.21}
Let $E$ and $F$ be two TMIB spaces on $\RR^d$ and let $\eta_1$ and $\eta_2$ be polynomially bounded weights on $\RR^d$. Assume that $\omega_F(x)=1$, $\forall x\in\RR^d$. Then $W(E,L^1_{\eta_1\eta_2})$ is equal to all of the following spaces with equivalent norms:
\begin{itemize}
\item[$(i)$] $\Modsq{W(F,L^{p_1}_{\eta_1})\hat{\otimes}_{\pi} \mathcal{F}W(E,L^{p_2}_{\eta_2})}$, with $p_1,p_2\in[1,\infty)$ satisfying $p^{-1}_1+p^{-1}_2\geq 1$;
\item[$(ii)$] $\Modsq{W(F,L^{\infty}_{\eta_1,0})\hat{\otimes}_{\pi} \mathcal{F}W(E,L^1_{\eta_2})}$;
\item[$(iii)$] $\Modsq{W(F,L^1_{\eta_1})\hat{\otimes}_{\pi} \mathcal{F}W(E,L^{\infty}_{\eta_2,0})}$.
\end{itemize}
\end{theorem}

\begin{remark}\label{rem-for-eq-norm-tra}
The condition $\omega_F(x)=1$, $\forall x\in\RR^d$, can be replaced with: $\omega_F$ bounded from above. When the latter holds true, Remark \ref{equiv-norm-for-bou-gr-f} implies that one can always find an equivalent norm on $F$ such that, with respect to this new norm, $\omega_F(x)=1$, $\forall x\in\RR^d$.
\end{remark}

For the proof of Theorem \ref{1.21}, we need the following preliminary result; its proof is a natural generalisation of the second part of the proof of Lemma \ref{dipiprxx3}.

\begin{lemma}\label{lemma-for-inc1}
Let $E$, $F$, $\eta_1$ and $\eta_2$ be as in Theorem \ref{1.21}. Then $W(E,L^1_{\eta_1\eta_2})$ is continuously included into the modulation space $\Modsq{W(F,L^1_{\eta_1})\hat{\otimes}_{\pi} \mathcal{F}W(E,L^1_{\eta_2})}$.
\end{lemma}

\begin{proof} Take $g\in\DD(\RR^d)$ such that $0\leq g\leq 1$ and $g=1$ on $[-2,2]^d$. Pick nonnegative $\varphi\in\DD((-1,1)^d)$ such that $\int_{\RR^d}\varphi(x)dx=1$, and define $\chi=\varphi*g\in\DD(\RR^d)$. Then $\chi$ is nonnegative and $\chi=1$ on $[-1,1]^d$ (and thus on $\supp\phi$ as well). Denote $\varphi_{\mathbf{k}}=T_{\mathbf{k}}\varphi\in\DD(\RR^d)$, $\mathbf{k}\in\ZZ^d$.\\
\indent Let $f\in W(E,L^1_{\eta_1\eta_2})$ and denote $\psi_{\mathbf{k}}=\mathcal{F}(f\phi_{\mathbf{k}})\in \mathcal{F}E$, $\mathbf{k}\in\ZZ^d$. Notice that
\beqs
\|f\phi_{\mathbf{k}}\phi_{\mathbf{m}}\|_{E}\leq \|f\phi_{\mathbf{k}}\|_{E}\|\phi_{\mathbf{m}}\|_{\mathcal{F}L^1_{\nu_E}}= \|f\phi_{\mathbf{k}}\|_{E} \|\phi\|_{\mathcal{F}L^1_{\nu_E}}
\eeqs
and thus
\beqs
\int_{\RR^d}\sum_{\mathbf{r}\in J_1}\|f\phi_{\mathbf{k}}\phi_{\mathbf{k}+\mathbf{r}}\|_E \theta_{[-1,1]^d}(t-\mathbf{k}-\mathbf{r})\eta_2(t)dt\leq C'\eta_2(\mathbf{k})\|f\phi_{\mathbf{k}}\|_{E}\|\phi\|_{\mathcal{F}L^1_{\nu_E}}.
\eeqs
We deduce $\psi_{\mathbf{k}}\in \mathcal{F}W(E,L^1_{\eta_2})$ with
\beqs
\|\psi_{\mathbf{k}}\|_{\mathcal{F}W(E,L^1_{\eta_2})}\leq C'\eta_2(\mathbf{k})\|f\phi_{\mathbf{k}}\|_{E}\|\phi\|_{\mathcal{F}L^1_{\nu_E}} ,\quad \forall \mathbf{k}\in\ZZ^d.
\eeqs
Notice that
\beqs
\|\varphi_{\mathbf{k}}\|_{W(F,L^1_{\eta_1})}&=& \int_{\RR^d}\sum_{\mathbf{m}\in\ZZ^d}\|T_{\mathbf{k}}(\varphi \phi_{\mathbf{m}-\mathbf{k}})\|_F\theta_{[-1,1]^d}(t-\mathbf{m})\eta_1(t)dt\\
&\leq& \int_{\RR^d}\sum_{\mathbf{m}\in\ZZ^d}\|\varphi \phi_{\mathbf{m}}\|_F\theta_{[-1,1]^d}(t-\mathbf{m}-\mathbf{k})\eta_1(t)dt\\
&\leq& C''\eta_1(\mathbf{k})\int_{\RR^d}\sum_{\mathbf{m}\in\ZZ^d}\|\varphi \phi_{\mathbf{m}}\|_F\theta_{[-1,1]^d}(t-\mathbf{m})(1+|t|)^{\tau}dt\\
&=& C''\eta_1(\mathbf{k})\|\varphi\|_{W(F,L^1_{(1+|\cdot|)^{\tau}})}.
\eeqs
The above estimates imply that the series $\sum_{\mathbf{k}\in\ZZ^d} \varphi_{\mathbf{k}}\otimes\psi_{\mathbf{k}}$ is absolutely summable in $W(F,L^1_{\eta_1})\hat{\otimes}_{\pi} \mathcal{F}W(E,L^1_{\eta_2})$ to some $\tilde{f}\in W(F,L^1_{\eta_1})\hat{\otimes}_{\pi} \mathcal{F}W(E,L^1_{\eta_2})$. Indeed,
\begin{multline*}
\sum_{\mathbf{k}\in\ZZ^d} \|\varphi_{\mathbf{k}}\otimes\psi_{\mathbf{k}}\|_{W(F,L^1_{\eta_1})\hat{\otimes}_{\pi} \mathcal{F}W(E,L^1_{\eta_2})}\\
\leq C_1\int_{\RR^d}\sum_{\mathbf{k}\in\ZZ^d} \|f\phi_{\mathbf{k}}\|_{E}\theta_{[-1,1]^d}(t-\mathbf{k}) \eta_1(\mathbf{k})\eta_2(\mathbf{k})dt\leq C_2\|f\|_{W(E,L^1_{\eta_1\eta_2})}.
\end{multline*}
Now, as $\phi_{\mathbf{k}}T_{\mathbf{k}}\chi=\phi_{\mathbf{k}}$, $\forall \mathbf{k}\in\ZZ^d$, we infer
\beqs
V^*_g(\tilde{f})=\sum_{\mathbf{k}\in\ZZ^d}V^*_g(\varphi_{\mathbf{k}}\otimes \psi_{\mathbf{k}})=\sum_{\mathbf{k}\in\ZZ^d} (\mathcal{F}^{-1}\psi_{\mathbf{k}})(\varphi_{\mathbf{k}}*g)=\sum_{\mathbf{k}\in\ZZ^d} f\phi_{\mathbf{k}}T_{\mathbf{k}}\chi=\sum_{\mathbf{k}\in\ZZ^d} f\phi_{\mathbf{k}}=f,
\eeqs
which, in view of \cite[Proposition 4.4]{dppv-3}, concludes the proof of the lemma.
\end{proof}

\begin{proof}[Proof of Theorem \ref{1.21}] Because of Lemma \ref{lemma-for-inc1} and \eqref{inc-s-fr}, the following continuous and dense inclusions hold true for all $1\leq p_1\leq p'_1<\infty$ and $1\leq p_2\leq p'_2<\infty$:
\begin{gather*}
W(E,L^1_{\eta_1\eta_2})\hookrightarrow \Modsq{W(F,L^{p_1}_{\eta_1})\hat{\otimes}_{\pi} \mathcal{F}W(E,L^{p_2}_{\eta_2})}\hookrightarrow \Modsq{W(F,L^{p'_1}_{\eta_1})\hat{\otimes}_{\pi} \mathcal{F}W(E,L^{p'_2}_{\eta_2})},\\
W(E,L^1_{\eta_1\eta_2})\hookrightarrow \Modsq{W(F,L^{p_1}_{\eta_1})\hat{\otimes}_{\pi} \mathcal{F}W(E,L^1_{\eta_2})}\hookrightarrow \Modsq{W(F,L^{\infty}_{\eta_1,0})\hat{\otimes}_{\pi} \mathcal{F}W(E,L^1_{\eta_2})},\\
W(E,L^1_{\eta_1\eta_2})\hookrightarrow \Modsq{W(F,L^1_{\eta_1})\hat{\otimes}_{\pi} \mathcal{F}W(E,L^{p_2}_{\eta_2})}\hookrightarrow \Modsq{W(F,L^1_{\eta_1})\hat{\otimes}_{\pi} \mathcal{F}W(E,L^{\infty}_{\eta_2,0})}.
\end{gather*}
Thus, it suffices to prove that $X\subseteq W(E,L^1_{\eta_1\eta_2})$ where $X$ is either the space in $(ii)$ or the space in $(iii)$ or $\Modsq{W(F,L^p_{\eta_1})\hat{\otimes}_{\pi} \mathcal{F}W(E,L^q_{\eta_2})}$ with $p\in(1,\infty)$ and $q$ its H\"older conjugate, since then the equivalence of the norms will follow from the open mapping theorem. We prove $\Modsq{W(F,L^p_{\eta_1})\hat{\otimes}_{\pi} \mathcal{F}W(E,L^q_{\eta_2})}\subseteq W(E,L^1_{\eta_1\eta_2})$, with $p,q\in(1,\infty)$ being H\"older conjugate as the proof when $X$ is either one of the spaces in $(ii)$ or $(iii)$ is analogous.\\
\indent Let $g\in\SSS(\RR^d)\backslash\{0\}$ be such that $\|g\|_{L^2(\RR^d)}=1$. Let $f\in\Modsq{W(F,L^p_{\eta_1})\hat{\otimes}_{\pi} \mathcal{F}W(E,L^q_{\eta_2})}$. Thus (see \cite[Theorem 6.4, p. 94]{Sch} and \cite[Theorem 6, p. 188]{kothe2})
\beqs
V_gf=\sum_{j=1}^{\infty} \lambda_j \varphi_j\otimes\psi_j
\eeqs
where $\varphi_j,\psi_j\in\SSS(\RR^d)$, $j\in\ZZ_+$, are such that $\{\varphi_j\}_{j\in\ZZ_+}$ and $\{\psi_j\}_{j\in\ZZ_+}$ are bounded subsets of $W(F,L^p_{\eta_1})$ and $\mathcal{F}W(E,L^q_{\eta_2})$ respectively; furthermore $\lambda_j$ are complex numbers such that $\sum_j|\lambda_j|<\infty$. We infer
\beq\label{sum-for-the-spp1}
f=V^*_gV_gf=\sum_{j=1}^{\infty} \lambda_j (\mathcal{F}^{-1}\psi_j)(\varphi_j*g),
\eeq
where the series is absolutely summable in $\Modsq{W(F,L^p_{\eta_1})\hat{\otimes}_{\pi} \mathcal{F}W(E,L^q_{\eta_2})}$. Pick $\chi\in\DD(\RR^d)$ such that $0\leq \chi\leq 1$, $\chi=1$ on $\supp\phi$ and $\supp\chi\subseteq [-2,2]^d$. Denote $\chi_{\mathbf{k}}=T_{\mathbf{k}}\chi$. As $\phi_{\mathbf{k}}\chi_{\mathbf{k}}=\phi_{\mathbf{k}}$, we infer
\beq\label{est-f-col1}
\|(\mathcal{F}^{-1}\psi_j)(\varphi_j*g)\phi_{\mathbf{k}}\|_{E}\leq \|(\mathcal{F}^{-1}\psi_j)\phi_{\mathbf{k}}\|_{E} \|(\varphi_j*g)\chi_{\mathbf{k}}\|_{\mathcal{F}L^1_{\nu_E}}.
\eeq
Let $m\in\ZZ_+$ be such that $(1+4\pi^2|\cdot|^2)^{-m}\nu_E\in L^1(\RR^d)$. We have
\beq
\|(\varphi_j*g)\chi_{\mathbf{k}}\|_{\mathcal{F}L^1_{\nu_E}}&=& \left\|\mathcal{F}^{-1}\left((1-\Delta)^m((\varphi_j*g)\chi_{\mathbf{k}})\right) (1+4\pi^2|\cdot|^2)^{-m}\nu_E\right\|_{L^1(\RR^d)}\nonumber \\
&\leq& \tilde{C}\sum_{|\alpha|\leq m}\sum_{\beta\leq 2\alpha}\frac{m!}{(m-|\alpha|)!\alpha!}{2\alpha\choose\beta} \left\|\mathcal{F}^{-1}\left((\varphi_j*\partial^{2\alpha-\beta}g) \partial^{\beta}\chi_{\mathbf{k}}\right)\right\|_{L^{\infty}(\RR^d)}\nonumber \\
&\leq& \tilde{C}\sum_{|\alpha|\leq m}\sum_{\beta\leq 2\alpha}\frac{m!}{(m-|\alpha|)!\alpha!}{2\alpha\choose\beta} \|(\varphi_j*\partial^{2\alpha-\beta}g) \partial^{\beta}\chi_{\mathbf{k}}\|_{L^1(\RR^d)}.\label{est-f-ter1}
\eeq
Observe that
\beqs
\|(\varphi_j*\partial^{2\alpha-\beta}g) \partial^{\beta}\chi_{\mathbf{k}}\|_{L^1(\RR^d)} \leq \sum_{\mathbf{m}\in\ZZ^d}\sum_{\mathbf{n}\in\ZZ^d} \left\|\left((\varphi_j\phi_{\mathbf{n}})*(\phi_{\mathbf{m}}\partial^{2\alpha-\beta}g)\right) \partial^{\beta}\chi_{\mathbf{k}}\right\|_{L^1(\RR^d)}.
\eeqs
Notice that $\supp((\varphi_j\phi_{\mathbf{n}})* (\phi_{\mathbf{m}}\partial^{2\alpha-\beta}g)) \subseteq \mathbf{n}+\mathbf{m}+[-2,2]^d$. Thus, if $\supp((\varphi_j\phi_{\mathbf{n}})* (\phi_{\mathbf{m}}\partial^{2\alpha-\beta}g))\cap \supp\chi_{\mathbf{k}}\neq \emptyset$ then $\mathbf{n}+\mathbf{m}\in \mathbf{k}+[-4,4]^d$. Since $\omega_F(x)=1$, $\forall x\in\RR^d$, Remark \ref{rem-f-con-map} yields
\beqs
\|(\varphi_j*\partial^{2\alpha-\beta}g) \partial^{\beta}\chi_{\mathbf{k}}\|_{L^1(\RR^d)}&\leq& \|\partial^{\beta}\chi\|_{L^1(\RR^d)} \sum_{\mathbf{m}\in\ZZ^d}\sum_{\mathbf{r}\in J_4} \left\|(\varphi_j\phi_{\mathbf{k}-\mathbf{m}+\mathbf{r}})* (\phi_{\mathbf{m}} \partial^{2\alpha-\beta}g)\right\|_{L^{\infty}(\RR^d)}\\
&\leq& \|\partial^{\beta}\chi\|_{L^1(\RR^d)} \sum_{\mathbf{m}\in\ZZ^d}\sum_{\mathbf{r}\in J_4} \|\varphi_j\phi_{\mathbf{k}-\mathbf{m}+\mathbf{r}}\|_F \|\phi_{\mathbf{m}}\partial^{2\alpha-\beta}g\|_{\check{F}'}.
\eeqs
Plugging this in \eqref{est-f-ter1} and employing Minkowski inequality, we infer
\begin{align*}
&\left(\sum_{\mathbf{k}\in\ZZ^d} \|(\varphi_j*g)\chi_{\mathbf{k}}\|_{\mathcal{F}L^1_{\nu_E}}^p \eta_1(\mathbf{k})^p\right)^{1/p}\\
&{}\quad\leq C_1\sum_{|\alpha|\leq m}\sum_{\beta\leq 2\alpha}\sum_{\mathbf{r}\in J_4} \left(\sum_{\mathbf{k}\in\ZZ^d}\left(\sum_{\mathbf{m}\in\ZZ^d} \|\varphi_j\phi_{\mathbf{k}-\mathbf{m}+\mathbf{r}}\|_F \eta_1(\mathbf{k}-\mathbf{m}+\mathbf{r}) \|\phi_{\mathbf{m}}\partial^{2\alpha-\beta}g\|_{\check{F}'} (1+|\mathbf{m}|)^{\tau}\right)^p \right)^{1/p}\\
&{}\quad\leq 9^dC_1\left(\sum_{\mathbf{n}\in\ZZ^d} \|\varphi_j\phi_{\mathbf{n}}\|_F^p\eta_1(\mathbf{n})^p\right)^{1/p} \sum_{|\alpha|\leq m} \sum_{\beta\leq 2\alpha}\sum_{\mathbf{m}\in\ZZ^d}\|\phi_{\mathbf{m}} \partial^{2\alpha-\beta}g\|_{\check{F}'}(1+|\mathbf{m}|)^{\tau}.
\end{align*}
By Remark \ref{rem-eq-nor-se-iis}, $(\sum_{\mathbf{n}\in\ZZ^d} \|\varphi_j\phi_{\mathbf{n}}\|_F^p\eta_1(\mathbf{n})^p)^{1/p} \leq C\|\varphi_j\|_{W(F,L^p_{\eta_1})}$. Since $\check{F}'$ is the strong dual of the TMIB space $\check{F}$, \eqref{inc-s-fr} implies that $\partial^{\gamma}g\in W(\check{F}',L^1_{(1+|\cdot|)^{\tau}})$, $\forall \gamma\in\NN^d$, which, in view of Remark \ref{rem-eq-nor-se-iis}, yields
\beqs
\sum_{\mathbf{m}\in\ZZ^d}\|\phi_{\mathbf{m}} \partial^{\gamma}g\|_{\check{F}'}(1+|\mathbf{m}|)^{\tau}<\infty,\quad \forall\gamma\in\NN^d.
\eeqs
Thus, \eqref{est-f-col1} gives
\begin{align*}
\|&(\mathcal{F}^{-1}\psi_j)(\varphi_j*g)\|_{W(E,L^1_{\eta_1\eta_2})}\\
&\leq C'_1\left(\sum_{\mathbf{k}\in\ZZ^d}\|(\mathcal{F}^{-1}\psi_j)\phi_{\mathbf{k}}\|_E^q \eta_2(\mathbf{k})^q\right)^{1/q}\left(\sum_{\mathbf{k}\in\ZZ^d} \|(\varphi_j*g)\chi_{\mathbf{k}}\|_{\mathcal{F}L^1_{\nu_E}}^p \eta_1(\mathbf{k})^p\right)^{1/p}\\
&\leq C'_2\left(\sum_{\mathbf{k}\in\ZZ^d}\|(\mathcal{F}^{-1}\psi_j)\phi_{\mathbf{k}}\|_E^q \eta_2(\mathbf{k})^q\right)^{1/q}\|\varphi_j\|_{W(F,L^p_{\eta_1})}.
\end{align*}
In view of Remark \ref{rem-eq-nor-se-iis}, this implies the series in \eqref{sum-for-the-spp1} is absolutely summable in $W(E,L^1_{\eta_1\eta_2})$ and the proof is complete.
\end{proof}

In Section \ref{sec-example} we discuss a number of interesting consequences of Theorem \ref{1.21}. Although this is obvious from the statement of Theorem \ref{1.21}, it is important to emphasise the following: the space $F$ (i.e. the local component of the first amalgam space) has no effect on the resulting modulation space; only the global behaviour of the first amalgam space matters.

\section{Modulation spaces associated to a class of completed $\epsilon$-tensor products of TMIB spaces}\label{sepsilon}

We now consider the modulation spaces associated to completed $\epsilon$-tensor products of a class of amalgam spaces. The main result of the section is the following; it is in the same spirit as Theorem \ref{1.21}.

\begin{theorem}\label{meinproposition1}
Let $E$ and $F$ be two TMIB spaces on $\RR^d$ and let $\eta_1$ and $\eta_2$ be polynomially bounded weights on $\RR^d$. Assume that $\omega_F(x)=1$, $\forall x\in\RR^d$. Then $W(E,L^{\infty}_{\eta_1\eta_2,0})$ is equal to all of the following spaces with equivalent norms:
\begin{itemize}
\item[$(i)$] $\Modsq{W(F,L^{p_1}_{\eta_1})\hat{\otimes}_{\epsilon} \mathcal{F}W(E,L^{p_2}_{\eta_2})}$, with $p_1^{-1}+p_2^{-1}\leq 1$, $p_1,p_2\in(1,\infty)$;
\item[$(ii)$] $\Modsq{W(F,L^{\infty}_{\eta_1,0})\hat{\otimes}_{\epsilon} \mathcal{F}W(E,L^{p_2}_{\eta_2})}$, with $p_2\in[1,\infty)$;
\item[$(iii)$] $\Modsq{W(F,L^{p_1}_{\eta_1})\hat{\otimes}_{\epsilon} \mathcal{F}W(E,L^{\infty}_{\eta_2,0})}$, with $p_1\in[1,\infty)$;
\item[$(iv)$] $\Modsq{W(F,L^{\infty}_{\eta_1,0})\hat{\otimes}_{\epsilon} \mathcal{F}W(E,L^{\infty}_{\eta_2,0})}$.
\end{itemize}
\end{theorem}

\begin{remark}
For the same reasons as in  Remark \ref{rem-for-eq-norm-tra}, the condition $\omega_F(x)=1$, $\forall x\in\RR^d$, can be replaced with: $\omega_F$ bounded from above.
\end{remark}

For the proof of Theorem \ref{meinproposition1}, we need the following two results.

\begin{lemma}\label{lem-sum-con1}
Let $E$, $F$, $\eta_1$ and $\eta_2$ be as in Theorem \ref{meinproposition1}. Then for any $\varphi\in\DD(\RR^d)$ and $e\in W(E,L^{\infty}_{\eta_1\eta_2,0})$, the family
\beq\label{famil-sum-c1}
\{(T_{\mathbf{k}}\varphi)\otimes (\phi_{\mathbf{k}}e)|\, \mathbf{k}\in\ZZ^d\}
\eeq
is summable in all of the following spaces
\begin{itemize}
\item[$(i)$] $W(F,L^{p_1}_{\eta_1})\hat{\otimes}_{\epsilon}W(E,L^{p_2}_{\eta_2})$, with $p_1^{-1}+p_2^{-1}\leq 1$, $p_1,p_2\in(1,\infty)$;
\item[$(ii)$] $W(F,L^{\infty}_{\eta_1,0})\hat{\otimes}_{\epsilon}W(E,L^{p_2}_{\eta_2})$, with $p_2\in[1,\infty)$;
\item[$(iii)$] $W(F,L^{p_1}_{\eta_1})\hat{\otimes}_{\epsilon}W(E,L^{\infty}_{\eta_2,0})$, with $p_1\in[1,\infty)$;
\item[$(iv)$] $W(F,L^{\infty}_{\eta_1,0})\hat{\otimes}_{\epsilon} W(E,L^{\infty}_{\eta_2,0})$.
\end{itemize}
\end{lemma}

\begin{proof} For $1\leq p_1\leq p'_1<\infty$ and $1\leq p_2\leq p'_2<\infty$, \eqref{inc-s-fr} gives the following continuous inclusion
\beqs
W(F,L^{p_1}_{\eta_1})\hat{\otimes}_{\epsilon}W(E,L^{p_2}_{\eta_2})\subseteq W(F,L^{p'_1}_{\eta_1})\hat{\otimes}_{\epsilon}W(E,L^{p'_2}_{\eta_2}).
\eeqs
Thus, to prove $(i)$ it suffices to verify it when $p_1=p\in(1,\infty)$ and $p_2=q\in(1,\infty)$ are H\"older conjugate.\\
\indent Let $\varphi\in \DD(\RR^d)$ and $e\in W(E,L^{\infty}_{\eta_1\eta_2,0})$. There is $k_0\in\ZZ_+$ such that $\supp\varphi\subseteq (-k_0,k_0)^d$. As $\eta_1$ and $\eta_2$ are polynomially bounded weights, there exists $C'\geq 1$ such that
\beq\label{est-for-pbww}
\eta_1(\mathbf{k}+\mathbf{s})\eta_2(\mathbf{k}+\mathbf{r})\leq C'\eta_1(\mathbf{k})\eta_2(\mathbf{k}),\quad \forall \mathbf{k}\in\ZZ^d,\, \forall \mathbf{r}\in J_1,\, \forall \mathbf{s}\in J_{k_0}.
\eeq
Let $\varepsilon>0$ and let $A'$ and $B'$ be equicontinuous subsets of $W(F,L^p_{\eta_1})'_b$ and $W(E,L^q_{\eta_2})'_b$ respectively. Because of Lemma \ref{dual-ams-s1}, without losing generality, we can assume
\beqs
A'&=&\{f'\in W(F',L^q_{1/\eta_1})|\, \|f'\|_{W(F',L^q_{1/\eta_1})}\leq C_1\},\quad \mbox{for some}\,\, C_1\geq 1,\\
B'&=&\{e'\in W(E',L^p_{1/\eta_2})|\, \|e'\|_{W(E',L^p_{1/\eta_2})}\leq C_2\},\quad \mbox{for some}\,\, C_2\geq 1.
\eeqs
As $e\in W(E,L^{\infty}_{\eta_1\eta_2,0})$, in view of Remark \ref{rem-eq-nor-se-iis}, there exists $n_0\in \ZZ_+$ such that
\beqs
\eta_1(\mathbf{k})\eta_2(\mathbf{k})\|\phi_{\mathbf{k}}e\|_E\leq \varepsilon\cdot (2\cdot3^d(2k_0+1)^dC'\tilde{C}^2C_1C_2\|\varphi\|_F)^{-1} =\varepsilon',\quad \forall \mathbf{k}\in\ZZ^d\backslash J_{n_0},
\eeqs
where $\tilde{C}\geq 1$ is the maximum of the two constants $C$ in Remark \ref{rem-eq-nor-se-iis} for the spaces $W(F',L^q_{1/\eta_1})$ and $W(E',L^p_{1/\eta_2})$. Then for finite $\Phi_1,\Phi_2\subseteq \ZZ^d$ with $J_{n_0}\subseteq \Phi_1\cap \Phi_2$, and $f'\in A'$ and $e'\in B'$, we infer
\begin{multline*}
\left|\left\langle f'\otimes e',\sum_{\mathbf{k}\in \Phi_1}(T_{\mathbf{k}}\varphi)\otimes(\phi_{\mathbf{k}}e)- \sum_{\mathbf{k}\in \Phi_2}(T_{\mathbf{k}}\varphi)\otimes(\phi_{\mathbf{k}}e)\right\rangle\right|\\
\leq \sum_{\mathbf{k}\in \Phi_1\backslash J_{n_0}}|\langle f',T_{\mathbf{k}}\varphi\rangle||\langle e',\phi_{\mathbf{k}}e\rangle|+\sum_{\mathbf{k}\in \Phi_2\backslash J_{n_0}}|\langle f',T_{\mathbf{k}}\varphi\rangle||\langle e',\phi_{\mathbf{k}}e\rangle|=I_1+I_2.
\end{multline*}
We estimate $I_1$ as follows
\beqs
I_1&\leq& \sum_{\mathbf{r}\in J_1}\sum_{\mathbf{s}\in J_{k_0}}\sum_{\mathbf{k}\in \Phi_1\backslash J_{n_0}}|\langle \phi_{\mathbf{k}+\mathbf{s}}f',T_{\mathbf{k}}\varphi\rangle||\langle \phi_{\mathbf{k}+\mathbf{r}}e',\phi_{\mathbf{k}}e\rangle|\\
&\leq& \sum_{\mathbf{r}\in J_1}\sum_{\mathbf{s}\in J_{k_0}}\sum_{\mathbf{k}\in \Phi_1\backslash J_{n_0}} \|\phi_{\mathbf{k}+\mathbf{s}}f'\|_{F'}\|T_{\mathbf{k}}\varphi\|_F \|\phi_{\mathbf{k}+\mathbf{r}}e'\|_{E'}\|\phi_{\mathbf{k}}e\|_E\\
&\leq& \varepsilon'\|\varphi\|_F\sum_{\mathbf{r}\in J_1}\sum_{\mathbf{s}\in J_{k_0}}\sum_{\mathbf{k}\in \Phi_1\backslash J_{n_0}} \|\phi_{\mathbf{k}+\mathbf{s}}f'\|_{F'}\|\phi_{\mathbf{k}+\mathbf{r}}e'\|_{E'} \eta_1(\mathbf{k})^{-1}\eta_2(\mathbf{k})^{-1}.
\eeqs
In view of \eqref{est-for-pbww} and Remark \ref{rem-eq-nor-se-iis}, we infer
\beqs
I_1\leq C'\varepsilon'\|\varphi\|_F\sum_{\mathbf{r}\in J_1}\sum_{\mathbf{s}\in J_{k_0}}\left(\sum_{\mathbf{k}\in \ZZ^d} \frac{\|\phi_{\mathbf{k}+\mathbf{s}}f'\|_{F'}^q} {\eta_1(\mathbf{k}+\mathbf{s})^q}\right)^{1/q} \left(\sum_{\mathbf{k}\in \ZZ^d} \frac{\|\phi_{\mathbf{k}+\mathbf{r}}e'\|_{E'}^p} {\eta_2(\mathbf{k}+\mathbf{r})^p}\right)^{1/p}\leq \varepsilon/2.
\eeqs
Analogously, $I_2\leq \varepsilon/2$ and consequently,
\beqs
\sup_{f'\in A',\,\,e'\in B'} \left|\left\langle f'\otimes e',\sum_{\mathbf{k}\in \Phi_1}(T_{\mathbf{k}}\varphi)\otimes(\phi_{\mathbf{k}}e)- \sum_{\mathbf{k}\in \Phi_2}(T_{\mathbf{k}}\varphi)\otimes(\phi_{\mathbf{k}}e)\right\rangle\right|\leq \varepsilon,
\eeqs
which completes the proof of $(i)$. To prove $(ii)$ it suffices to verify the claim for $W(F,L^{\infty}_{\eta_1,0})\hat{\otimes}_{\epsilon}W(E,L^1_{\eta_2})$ as this spaces is continuously included into the rest of the spaces in $(ii)$; the proof of this fact can be performed in an analogous fashion as for $(i)$. For $(iii)$, one reasons in the same way as for $(ii)$, and finally, $(iv)$ follows from either one of $(i)$, $(ii)$ or $(iii)$.
\end{proof}

\begin{lemma}\label{inc-eps-top1}
Let $E$, $F$, $\eta_1$ and $\eta_2$ be as in Theorem \ref{meinproposition1}. Then $W(E,L^{\infty}_{\eta_1\eta_2,0})$ is continuously included into $\Modsq{X}$ where $X$ is either one of the following TMIB spaces:
\begin{itemize}
\item[$(a)$] $X=W(F,L^{p_1}_{\eta_1})\hat{\otimes}_{\epsilon}\mathcal{F}W(E,L^{p_2}_{\eta_2})$, with $p_1^{-1}+p_2^{-1}\leq 1$, $p_1,p_2\in(1,\infty)$;
\item[$(b)$] $X=W(F,L^{\infty}_{\eta_1,0})\hat{\otimes}_{\epsilon}\mathcal{F}W(E,L^{p_2}_{\eta_2})$, with $p_2\in[1,\infty)$;
\item[$(c)$] $X=W(F,L^{p_1}_{\eta_1})\hat{\otimes}_{\epsilon}\mathcal{F}W(E,L^{\infty}_{\eta_2,0})$, with $p_1\in[1,\infty)$;
\item[$(d)$] $X=W(F,L^{\infty}_{\eta_1,0})\hat{\otimes}_{\epsilon} \mathcal{F}W(E,L^{\infty}_{\eta_2,0})$.
\end{itemize}
\end{lemma}

\begin{proof} Take $g\in\DD(\RR^d)$ such that $0\leq g\leq 1$ and $g=1$ on $[-2,2]^d$. Pick nonnegative $\varphi\in\DD((-1,1)^d)$ such that $\int_{\RR^d}\varphi(x)dx=1$, and define $\chi=\varphi*g\in\DD(\RR^d)$. Then $\chi$ is nonnegative and $\chi=1$ on $[-1,1]^d$ (and thus on $\supp\phi$ as well). Set $\varphi_{\mathbf{k}}=T_{\mathbf{k}}\varphi$, $\mathbf{k}\in\ZZ^d$. Let $f\in W(E,L^{\infty}_{\eta_1\eta_2,0})$. Since $\operatorname{Id}\hat{\otimes}_{\epsilon}\mathcal{F}$ is continuous as a mapping from the spaces in $(i)$, $(ii)$, $(iii)$ and $(iv)$ in Lemma \ref{lem-sum-con1} into the spaces in $(a)$, $(b)$, $(c)$ and $(d)$ respectively, Lemma \ref{lem-sum-con1} implies that $\sum_{\mathbf{k}\in\ZZ^d} \varphi_{\mathbf{k}}\otimes \mathcal{F}(\phi_{\mathbf{k}}f)$ is summable in all of the spaces in $(a)$, $(b)$, $(c)$ and $(d)$ to an $\tilde{f}$ which belongs in all of these spaces. As
\beqs
V^*_g(\tilde{f})=\sum_{\mathbf{k}\in\ZZ^d}V^*_g(\varphi_{\mathbf{k}}\otimes \mathcal{F}(\phi_{\mathbf{k}}f))=\sum_{\mathbf{k}\in\ZZ^d} (\phi_{\mathbf{k}}f)(\varphi_{\mathbf{k}}*g)=\sum_{\mathbf{k}\in\ZZ^d} f\phi_{\mathbf{k}}T_{\mathbf{k}}\chi=\sum_{\mathbf{k}\in\ZZ^d} f\phi_{\mathbf{k}}=f,
\eeqs
\cite[Proposition 4.4]{dppv-3} yields that $W(E,L^{\infty}_{\eta_1\eta_2,0})\subseteq \Modsq{X}$ with $X$ either one of the TMIB spaces in $(a)$, $(b)$, $(c)$ or $(d)$. The fact that the inclusions are continuous follows from the closed graph theorem as all of these spaces are continuously included into $\SSS'(\RR^d)$.
\end{proof}

\begin{proof}[Proof of Theorem \ref{meinproposition1}] Since all of the spaces in $(i)$, $(ii)$ and $(iii)$ are continuously included into the space in $(iv)$, in view of Lemma \ref{inc-eps-top1} it is enough to prove that
\beqs
\Modsq{W(F,L^{\infty}_{\eta_1,0})\hat{\otimes}_{\epsilon} \mathcal{F}W(E,L^{\infty}_{\eta_2,0})}\subseteq W(E,L^{\infty}_{\eta_1,\eta_2,0})
\eeqs
as the continuity of the inclusion will follow from the open mapping theorem. By \cite[Proposition 4.4]{dppv-3} it suffices to show that
\beq\label{inc-t-prr1}
V^*_g(W(F,L^{\infty}_{\eta_1,0})\hat{\otimes}_{\epsilon} \mathcal{F}W(E,L^{\infty}_{\eta_2,0}))\subseteq W(E,L^{\infty}_{\eta_1,\eta_2,0}),\quad \mbox{for some}\,\,g\in \DD(\RR^d)\backslash\{0\}.
\eeq
Fix $g\in\DD(\RR^d)\backslash\{0\}$ and pick $n_0\in \ZZ_+$ such that $\supp g\subseteq (-n_0,n_0)^d$. Let $u\in \SSS(\RR^d)\otimes \SSS(\RR^d)$; hence
\beqs
u=\sum_{j=1}^s \varphi_j\otimes \psi_j,\quad \mbox{with}\,\, \varphi_j,\psi_j\in\SSS(\RR^d),\,j=1,\ldots,s.
\eeqs
Consequently
\beqs
V^*_gu=\sum_{j=1}^s (\mathcal{F}^{-1}\psi_j)(\varphi_j*g).
\eeqs
Set $\tilde{g}(x,y)=g(y-x)$; clearly $\tilde{g}\in\DD_{L^{\infty}}(\RR^{2d})$. Let $e'\in E'$ be arbitrary but fixed. Denoting by $1_{\RR^d}$ the functions on $\RR^d$ which is identically equal to $1$, we infer
\beq
\langle e', \phi_{\mathbf{k}}V^*_gu\rangle&=&\left\langle 1_{\RR^d}\otimes (\phi_{\mathbf{k}}e'),\tilde{g}\sum_{j=1}^s \varphi_j\otimes (\mathcal{F}^{-1}\psi_j)\right\rangle=\langle 1_{\RR^d}\otimes (\phi_{\mathbf{k}}e'),\tilde{g}(\operatorname{Id}\otimes \mathcal{F}^{-1})(u)\rangle\nonumber\\
&=& \sum_{\mathbf{m}\in\ZZ^d}\langle \tilde{g}(\phi_{\mathbf{m}}\otimes (\phi_{\mathbf{k}}e')),(\operatorname{Id}\otimes \mathcal{F}^{-1})(u)\rangle.
\eeq
It is straightforward to check that if $\mathbf{m}\not\in \mathbf{k}+J_{n_0+2}$ then $\tilde{g}(\phi_{\mathbf{m}}\otimes (\phi_{\mathbf{k}}e'))=0$. Pick $\chi\in\DD(\RR^d)$ such that $0\leq \chi\leq 1$ and $\chi=1$ on $\supp\phi$. Denoting $\chi_{\mathbf{n}}=T_{\mathbf{n}}\chi$, $\mathbf{n}\in\ZZ^d$, we infer
\beq\label{est-du-fbb}
|\langle e', \phi_{\mathbf{k}}V^*_gu\rangle|\leq \sum_{\mathbf{r}\in J_{n_0+2}}|\langle \phi_{\mathbf{k}+\mathbf{r}}\otimes (\phi_{\mathbf{k}}e'),\tilde{g}(\chi_{\mathbf{k}+\mathbf{r}}\otimes 1_{\RR^d}) (\operatorname{Id}\otimes \mathcal{F}^{-1})(u)\rangle|.
\eeq
Denote by $X$ the TMIB space $W(F,L^{\infty}_{\eta_1,0})\hat{\otimes}_{\epsilon} W(E,L^{\infty}_{\eta_2,0})$; its norm is given by
\beqs
\|v\|_X=\sup_{\|f'\|_{W(F,L^{\infty}_{\eta_1,0})'_b}\leq 1}\sup_{\|e'\|_{W(E,L^{\infty}_{\eta_2,0})'_b}\leq 1}|\langle f'\otimes e',v\rangle|,\quad v\in X.
\eeqs
For any $\mathbf{n}\in \ZZ^d$, we have
\beqs
\sum_{\mathbf{m}\in\ZZ^d}\frac{\|\phi_{\mathbf{n}}\phi_{\mathbf{m}}\|_{F'}} {\eta_1(\mathbf{m})}\leq \sum_{\mathbf{r}\in J_1}\frac{\|\phi_{\mathbf{n}}\|_{F'} \|\phi_{\mathbf{n}+\mathbf{r}}\|_{\mathcal{F}L^1_{\nu_F}}}{\eta_1(\mathbf{n}+\mathbf{r})} \leq C_1\|\phi\|_{\mathcal{F}L^1_{\nu_F}}\|\phi\|_{F'}/\eta_1(\mathbf{n}).
\eeqs
Consequently, Lemma \ref{dual-ams-s1} $(iii)$ together with Remark \ref{rem-eq-nor-se-iis} imply that
\beqs
\|\phi_{\mathbf{n}}\|_{W(F,L^{\infty}_{\eta_1,0})'_b}\leq C'_1/\eta_1(\mathbf{n}),\quad \forall \mathbf{n}\in\ZZ^d.
\eeqs
Similarly, for any $\mathbf{n}\in\ZZ^d$,
\beqs
\sum_{\mathbf{m}\in\ZZ^d}\frac{\|\phi_{\mathbf{n}}\phi_{\mathbf{m}}e'\|_{E'}} {\eta_2(\mathbf{m})}\leq \sum_{\mathbf{r}\in J_1}\frac{\|\phi_{\mathbf{n}}\|_{\mathcal{F}L^1_{\nu_E}} \|\phi_{\mathbf{n}+\mathbf{r}}\|_{\mathcal{F}L^1_{\nu_E}}\|e'\|_{E'}} {\eta_2(\mathbf{n}+\mathbf{r})}\leq C_2\|\phi\|_{\mathcal{F}L^1_{\nu_E}}^2\|e'\|_{E'}/\eta_2(\mathbf{n}),
\eeqs
and, again, Lemma \ref{dual-ams-s1} $(iii)$ and Remark \ref{rem-eq-nor-se-iis} imply $\phi_{\mathbf{n}}e'\in W(E,L^{\infty}_{\eta_2,0})'_b$ with
\beqs
\|\phi_{\mathbf{n}}e'\|_{W(E,L^{\infty}_{\eta_2,0})'_b}\leq C'_2\|e'\|_{E'}/\eta_2(\mathbf{n}),\quad \forall \mathbf{n}\in\ZZ^d.
\eeqs
Plugging these bounds in \eqref{est-du-fbb}, we infer
\beqs
|\langle e', \phi_{\mathbf{k}}V^*_gu\rangle|\leq \frac{C_3\|e'\|_{E'}}{\eta_1(\mathbf{k})\eta_2(\mathbf{k})} \sum_{\mathbf{r}\in J_{n_0+2}}\|\tilde{g}(\chi_{\mathbf{k}+\mathbf{r}}\otimes 1_{\RR^d}) (\operatorname{Id}\otimes \mathcal{F}^{-1})(u)\|_X,\quad \forall \mathbf{k}\in\ZZ^d.
\eeqs
As $e'\in E'$ was arbitrary, we deduce
\beqs
\eta_1(\mathbf{k})\eta_2(\mathbf{k})\|\phi_{\mathbf{k}}V^*_gu\|_E\leq C_3 \sum_{\mathbf{r}\in J_{n_0+2}}\|\tilde{g}(\chi_{\mathbf{k}+\mathbf{r}}\otimes 1_{\RR^d}) (\operatorname{Id}\otimes \mathcal{F}^{-1})(u)\|_X,\quad \forall \mathbf{k}\in\ZZ^d.
\eeqs
Notice that $\tilde{g}(\chi_{\mathbf{k}+\mathbf{r}}\otimes 1)\in\SSS(\RR^{2d})$. Our immediate goal is to estimate the norm of the latter in $\mathcal{F}L^1_{\nu_X}$. In view of \cite[Theorem 3.6]{dppv-3}, we have
\beqs
\nu_X=\nu_{W(F,L^{\infty}_{\eta_1,0})}\otimes\nu_{W(E,L^{\infty}_{\eta_2,0})}\leq \nu_F\otimes\nu_E;
\eeqs
the very last inequality is easy to verify. As
\beqs
\mathcal{F}^{-1}(\tilde{g}(\chi_{\mathbf{k}+\mathbf{r}}\otimes 1))(t,\xi)=e^{2\pi i (\mathbf{k}+\mathbf{r})(t+\xi)}\mathcal{F}^{-1}\chi(t+\xi)\mathcal{F}^{-1}g(\xi),
\eeqs
we deduce $\|\tilde{g}(\chi_{\mathbf{k}+\mathbf{r}}\otimes 1)\|_{\mathcal{F}L^1_{\nu_X}}\leq \|\chi\|_{\mathcal{F}L^1}\|g\|_{\mathcal{F}L^1_{\nu_E}}$, for all $\mathbf{k}\in\ZZ^d$, $\mathbf{r}\in J_{n_0+2}$. Hence
\beqs
\eta_1(\mathbf{k})\eta_2(\mathbf{k})\|\phi_{\mathbf{k}}V^*_gu\|_E\leq C_4\|(\operatorname{Id}\otimes \mathcal{F}^{-1})(u)\|_X,\quad \forall \mathbf{k}\in\ZZ^d.
\eeqs
Since $\operatorname{Id}\hat{\otimes}_{\epsilon}\mathcal{F}^{-1}: W(F,L^{\infty}_{\eta_1,0})\hat{\otimes}_{\epsilon} \mathcal{F}W(E,L^{\infty}_{\eta_2,0})\rightarrow X$ is continuous, Remark \ref{rem-eq-nor-se-iis} yields
\beqs
\|V^*_gu\|_{W(E,L^{\infty}_{\eta_1\eta_2,0})}\leq C_5\|u\|_{W(F,L^{\infty}_{\eta_1,0})\hat{\otimes}_{\epsilon} \mathcal{F}W(E,L^{\infty}_{\eta_2,0})},\quad \forall u\in \SSS(\RR^d)\otimes \SSS(\RR^d).
\eeqs
The validity of \eqref{inc-t-prr1} follows from the density of $\SSS(\RR^d)\otimes \SSS(\RR^d)$ in $W(F,L^{\infty}_{\eta_1,0})\hat{\otimes}_{\epsilon} \mathcal{F}W(E,L^{\infty}_{\eta_2,0})$.
\end{proof}

In the next section we derive a number of interesting corollaries of Theorem \ref{meinproposition1}. The same observation we made at the very end of Section \ref{sub-sec-pi-prod} holds equally well for Theorem \ref{meinproposition1}: the space $F$ has no effect on the resulting modulation space; only the global behaviour of the first amalgam space matters.

\section{Consequences of Theorem \ref{1.21} and Theorem \ref{meinproposition1}}\label{sec-example}

We devote this section to deriving a number of consequences of the two main results. We start by considering the modulation spaces associated to tensor products of $L^p$ spaces.

\begin{corollary}\label{cor12}
The following assertions hold true.
\begin{itemize}
\item[$(a)$] $\Modsq{L^{p_1}\hat{\otimes}_{\pi} L^{p_2}}=W(\mathcal F L^{p_2},L^1)$, for all $1\leq p_1\leq p_2\leq 2$.
\item[$(b)$] $\Modsq{L^{p_1}\hat{\otimes}_{\varepsilon} L^{p_2}}=\Modsq{\mathcal{C}_0\hat{\otimes}_{\epsilon} L^{p_2}}= W(\mathcal F L^{p_2},L^{\infty}_0)$, for all $2\leq p_2\leq p_1<\infty$.
\item[$(c)$] Let $\eta_1$ and $\eta_2$ be two polynomially bounded weights on $\RR^d$. Then
    \beqs
    \Modsq{\mathcal{C}_{\eta_1\otimes \eta_2,0}(\RR^{2d})} = \Modsq{\mathcal{C}_{\eta_1,0}(\RR^d)\hat{\otimes}_{\epsilon} \mathcal{C}_{\eta_2,0}(\RR^d)}= W(\mathcal{F}\mathcal{C}_{\check{\eta}_2,0},L^{\infty}_{\eta_1,0}).
    \eeqs
\end{itemize}
\end{corollary}

\begin{proof} $(a)$ If $p_2=1$ then $p_1=1$ and the claim follows from $L^1(\RR^d)\hat{\otimes}_{\pi}L^1(\RR^d)=L^1(\RR^{2d})$ (see \cite[Section 2.3]{ryan}), since $\Modsq{L^1(\RR^{2d})}=M^{1,1}=W(\mathcal{F}L^1,L^1)$.\\
\indent By \cite[Lemma 1.2 $iv)$]{fe81} we have $L^p=W(L^p,L^p)\subseteq W(\mathcal{F}L^q,L^p)$, for all $p\in[1,2]$, where $q$ is the H\"older conjugate to $p$. Consequently, \cite[Theorem 3.2]{fe81} gives $\mathcal{F}^{-1}L^p=\mathcal{F}L^p\subseteq W(\mathcal{F}L^p,L^q)$ which yields the continuous inclusion:
\beq\label{ste-for-s}
L^p\subseteq \mathcal{F}W(\mathcal{F}L^p,L^q),\quad \forall p\in(1,2].
\eeq
Let $1\leq p_1\leq p_2\leq 2$ with $p_2>1$ and denote by $q_1$ and $q_2$ the H\"older conjugate indexes to $p_1$ and $p_2$ respectively. As $W(\mathcal{F}L^{p_2},L^1)\subseteq \mathcal{F}L^{p_2}=\mathcal{F}^{-1}L^{p_2}$ (by \eqref{inc-am-sp-n}), employing \eqref{ste-for-s}, we deduce
\beqs
W(L^{p_1},L^{p_1})\hat{\otimes}_{\pi} \mathcal{F}W(\mathcal{F}L^{p_2},L^1)\subseteq L^{p_1}\hat{\otimes}_{\pi} L^{p_2}\subseteq W(L^{p_1},L^{p_1})\hat{\otimes}_{\pi} \mathcal{F}W(\mathcal{F}L^{p_2},L^{q_2}).
\eeqs
Now the claim follows from Theorem \ref{1.21}, since $p_1^{-1}+q_2^{-1}\geq 1$ (equivalent to $p_1\leq p_2$).\\
\indent $(b)$ Let $2\leq p_2\leq p_1<\infty$ and denote by $q_1$ and $q_2$ the H\"older conjugate indexes to $p_1$ and $p_2$ respectively. Employing \cite[Theorem 3.2]{fe81} and \cite[Lemma 1.2 $iv)$]{fe81}, we infer the continuous inclusions:
\beqs
\mathcal{F}W(\mathcal{F}L^{p_2},L^{q_2})\subseteq W(\mathcal{F}L^{q_2},L^{p_2})\subseteq W(L^{p_2},L^{p_2})=L^{p_2}.
\eeqs
Since $\mathcal{F}^{-1}L^{p_2}=\mathcal{F}L^{p_2}\subseteq W(\mathcal{F}L^{p_2},L^{\infty}_0)$ (by \eqref{inc-am-sp-n}), we deduce the continuous inclusions:
\beqs
W(L^{p_1},L^{p_1})\hat{\otimes}_{\epsilon} \mathcal{F}W(\mathcal{F}L^{p_2}, L^{q_2}) \subseteq L^{p_1}\hat{\otimes}_{\epsilon} L^{p_2}\subseteq W(L^{p_1},L^{p_1})\hat{\otimes}_{\epsilon} \mathcal{F}W(\mathcal{F}L^{p_2},L^{\infty}_0).
\eeqs
Now Theorem \ref{meinproposition1} gives $\Modsq{L^{p_1}\hat{\otimes}_{\epsilon} L^{p_2}}=W(\mathcal{F}L^{p_2},L^{\infty}_0)$, since $p_1^{-1}+q_2^{-1}\leq 1$ (equivalent to $p_2\leq p_1$). The proof of $\Modsq{\mathcal{C}_0\hat{\otimes}_{\epsilon} L^{p_2}}=W(\mathcal{F}L^{p_2},L^{\infty}_0)$ is analogous since $\mathcal{C}_0=W(\mathcal{C}_0,L^{\infty}_0)$.\\
\indent $(c)$ The first equality follows from $\mathcal{C}_{\eta_1,0}(\RR^d)\hat{\otimes}_{\epsilon}\mathcal{C}_{\eta_2,0}(\RR^d)= \mathcal{C}_{\eta_1\otimes\eta_2,0}(\RR^{2d})$. As (cf. \eqref{inc-am-sp-n})
\beqs
W(\mathcal{F}\mathcal{C}_{\check{\eta}_2,0},L^1)\subseteq \mathcal{F}\mathcal{C}_{\check{\eta}_2,0}\subseteq W(\mathcal{F}\mathcal{C}_{\check{\eta}_2,0},L^{\infty}_0),
\eeqs
$\mathcal{F}\mathcal{C}_{\check{\eta}_2,0}=\mathcal{F}^{-1}\mathcal{C}_{\eta_2,0}$ and $\mathcal{C}_{\eta_1,0}=W(\mathcal{C}_0,L^{\infty}_{\eta_1,0})$, we have the continuous inclusions:
\beqs
W(\mathcal{C}_0,L^{\infty}_{\eta_1,0})\hat{\otimes}_{\epsilon} \mathcal{F}W(\mathcal{F}\mathcal{C}_{\check{\eta}_2,0},L^1)\subseteq \mathcal{C}_{\eta_1,0}\hat{\otimes}_{\epsilon}\mathcal{C}_{\eta_2,0}\subseteq W(\mathcal{C}_0,L^{\infty}_{\eta_1,0})\hat{\otimes}_{\epsilon} \mathcal{F}W(\mathcal{F}\mathcal{C}_{\check{\eta}_2,0},L^{\infty}_0)
\eeqs
and the claim follows from Theorem \ref{meinproposition1}.
\end{proof}

\begin{remark}\label{rem-for-lpp}
In view of Lemma \ref{ident-mod-amalg}, Corollary \ref{cor12} $(a)$ implies
\beq
\Modsq{L^{p_1}\hat{\otimes}_{\pi} L^{p_2}}=\mathcal{F}M^{p_2,1},\quad \mbox{for all}\,\, 1\leq p_1\leq p_2\leq 2.
\eeq
\end{remark}

Next, we derive two rather general corollaries of Theorem \ref{1.21} and Theorem \ref{meinproposition1} which can be employed for deducing a variety of special instances.

\begin{corollary}\label{cor-for-inbetwe-spa}
Let $E$ and $F$ be two TMIB spaces on $\RR^d$ such that $\omega_F(x)=1$, $\forall x\in\RR^d$, and let $\eta$ be a polynomially bounded weight on $\RR^d$. Assume that $E$ and $F$ satisfy the following conditions for some $p\in[1,\infty)$:
\begin{itemize}
\item[$(a)$] there exists a TMIB space $E_1$ on $\RR^d$ such that $W(E_1,L^1_{\eta})\subseteq \mathcal{F}^{-1}E\subseteq W(E_1,L^q_{\eta})$ with $q$ the H\"older conjugate index to $p$;
\item[$(b)$] $F\subseteq W(F_1,L^p)$ for some TMIB space $F_1$ on $\RR^d$ satisfying $\omega_{F_1}(x)=1$, $\forall x\in\RR^d$.
\end{itemize}
Then $\Modsq{F\hat{\otimes}_{\pi} E}=W(E_1,L^1_{\eta})$ with equivalent norms.\\
\indent In particular, if $E$ satisfies $(a)$ with $E_1=\mathcal{F}L^p_{\eta_1}$ for some polynomially bounded weight $\eta_1$, then $\Modsq{F\hat{\otimes}_{\pi} E}=W(\mathcal{F}L^p_{\eta_1},L^1_{\eta})$ with equivalent norms.
\end{corollary}

\begin{proof} Since all of the spaces are continuously included into $\SSS'(\RR^d)$, the closed graph theorem implies that all of the inclusions in $(a)$ and $(b)$ are continuous. Assume first $p\in(1,\infty)$ and hence $q\in(1,\infty)$. As $W(F,L^1)\subseteq F$ (by \eqref{inc-am-sp-n}), we infer the continuous inclusions:
\beqs
W(F,L^1)\hat{\otimes}_{\pi} \mathcal{F}W(E_1,L^1_{\eta})\subseteq F\hat{\otimes}_{\pi} E \subseteq W(F_1,L^p)\hat{\otimes}_{\pi} \mathcal{F}W(E_1,L^q_{\eta}),
\eeqs
and the result follows from Theorem \ref{1.21}.\\
\indent When $p=1$, we can not directly apply Theorem \ref{1.21}; notice that $W(E_1, L^{\infty}_{\eta})$ is not necessarily even a TMIB space ($\SSS(\RR^d)$ may fail to be dense in it). However, as $\SSS(\RR^d)$ is dense in $\mathcal{F}^{-1}E$, we deduce that $\mathcal{F}^{-1}E$ is continuously included into the closure of $\SSS(\RR^d)$ in $W(E_1, L^{\infty}_{\eta})$. The latter space is $W(E_1,L^{\infty}_{\eta,0})$. Indeed, $\SSS(\RR^d)$ is dense in $W(E_1,L^{\infty}_{\eta,0})$ and $W(E_1,L^{\infty}_{\eta,0})$ is a closed subspace of $W(E_1,L^{\infty}_{\eta})$. Thus, the second inclusion in $(a)$ can be strengthen as $\mathcal{F}^{-1}E\subseteq W(E_1,L^{\infty}_{\eta,0})$. Now, the rest of the proof can be done in an analogous fashion as for the case $p\in(1,\infty)$.
\end{proof}

The analogous result for the $\epsilon$-tensor product is the following; the proof is similar to the proof of Corollary \ref{cor-for-inbetwe-spa} and we omit it (of course, now one applies Theorem \ref{meinproposition1}).

\begin{corollary}\label{prs-t-tmibspa}
Let $E$ and $F$ be two TMIB spaces on $\RR^d$ such that $\omega_F(x)=1$, $\forall x\in\RR^d$, and let $\eta$ be a polynomially bounded weight on $\RR^d$. Assume that $E$ and $F$ satisfy the following conditions for some $p\in(1,\infty)$:
\begin{itemize}
\item[$(a)$] there exists a TMIB space $E_1$ on $\RR^d$ such that $W(E_1,L^q_{\eta})\subseteq \mathcal{F}^{-1}E\subseteq W(E_1,L^{\infty}_{\eta,0})$ with $q$ the H\"older conjugate index to $p$;
\item[$(b)$] $W(F_1,L^p)\subseteq F$ for some TMIB space $F_1$ on $\RR^d$ satisfying $\omega_{F_1}(x)=1$, $\forall x\in\RR^d$.
\end{itemize}
Then $\Modsq{F\hat{\otimes}_{\epsilon} E}=W(E_1,L^{\infty}_{\eta,0})$ with equivalent norms.\\
\indent In particular, if $E$ satisfies $(a)$ with $E_1=\mathcal{F}L^p_{\eta_1}$ for some polynomially bounded weight $\eta_1$, then $\Modsq{F\hat{\otimes}_{\epsilon} E}=W(\mathcal{F}L^p_{\eta_1},L^{\infty}_{\eta,0})$ with equivalent norms.
\end{corollary}

As a consequence of these corollaries, we have the following result.

\begin{corollary}
Let $\eta_1$ and $\eta_2$ be two polynomially bounded weights on $\RR^d$ which satisfy $\eta_1=\check{\eta}_1$ and $\eta_2=\check{\eta}_2$.
\begin{itemize}
\item[$(i)$] For any $1\leq p_2\leq p_1<\infty$ satisfying $p_1^{-1}+p_2^{-1}\geq 1$,
\beqs
\Modsq{M^{p_1,p_2}\hat{\otimes}_{\pi} M^{p_1,p_2}_{\eta_1\otimes\eta_2}}=W(\mathcal{F}L^{p_1}_{\eta_1},L^1_{\eta_2})\quad \mbox{with equivalent norms.}
\eeqs
\item[$(ii)$] For any $1< p_1\leq p_2<\infty$ satisfying $p_1^{-1}+p_2^{-1}\leq 1$,
\beqs
\Modsq{M^{p_1,p_2}\hat{\otimes}_{\epsilon} M^{p_1,p_2}_{\eta_1\otimes\eta_2}}=W(\mathcal{F}L^{p_1}_{\eta_1},L^{\infty}_{\eta_2,0})\quad \mbox{with equivalent norms.}
\eeqs
\end{itemize}
\end{corollary}

\begin{proof} For any $g\in\SSS(\RR^d)\backslash\{0\}$ and $f\in\SSS'(\RR^d)$, \cite[Lemma 3.1.1, p. 39]{Grochenig} yields
\beq\label{stfte-for-m}
V_g(\mathcal{F}^{-1}f)(x,\xi)=V_{\mathcal{F}\mathcal{F}g}(\mathcal{F}f)(-x,-\xi),\quad \mbox{for all}\,\, x,\xi\in\RR^d.
\eeq
Thus $\mathcal{F}M^{p_1,p_2}_{\eta_1\otimes\eta_2}= \mathcal{F}^{-1}M^{p_1,p_2}_{\eta_1\otimes\eta_2}$ (because of the condition on the weights).\\
\indent To prove $(i)$, we apply Corollary \ref{cor-for-inbetwe-spa} with $E=M^{p_1,p_2}_{\eta_1\otimes\eta_2}$, $F=M^{p_1,p_2}$ and $\eta=\eta_2$. We claim the conditions in Corollary \ref{cor-for-inbetwe-spa} are satisfied with $E_1=\mathcal{F}L^{p_1}_{\eta_1}$, $F_1=\mathcal{F}L^{p_2}$ and $p=p_1$. Indeed, \eqref{inc-s-fr}, Lemma \ref{ident-mod-amalg} and $p_1^{-1}+p_2^{-1}\geq 1$ (which is equivalent to $p_2\leq q$) imply
\beqs
W(E_1,L^1_{\eta})\subseteq W(\mathcal{F}L^{p_1}_{\eta_1},L^{p_2}_{\eta_2})=\mathcal{F}^{-1}E\subseteq W(E_1,L^q_{\eta}).
\eeqs
Hence, the condition $(a)$ is satisfied. The above together with Lemma \ref{ident-mod-amalg} and \cite[Theorem 3.2]{fe81} (since $p_2\leq p_1$) gives $F= \mathcal{F}W(\mathcal{F}L^{p_1},L^{p_2})\subseteq W(F_1,L^p)$ which is the condition $(b)$. Now the claim follows from Corollary \ref{cor-for-inbetwe-spa}.\\
\indent The proof of $(ii)$ is analogous and we omit it (of course, one has to use Corollary \ref{prs-t-tmibspa} instead).
\end{proof}

We can slightly improve this result when $p_1$ and $p_2$ are the same in the first modulation space.

\begin{corollary}
Let $\eta_1$, $\eta_2$ and $\eta_3$ be polynomially bounded weights on $\RR^d$. Assume that $\eta_1=\check{\eta}_1$.
\begin{itemize}
\item[$(i)$] For any $p_1,p_2\in[1,\infty)$ satisfying $p_1^{-1}+p_2^{-1}\geq 1$ and any $p_3\in[1,\infty)$,
\beqs
\Modsq{M^{p_1,p_1}_{\eta_1\otimes\eta_1}\hat{\otimes}_{\pi} M^{p_3,p_2}_{\eta_3\otimes\eta_2}}= W(\mathcal{F}L^{p_3}_{\check{\eta}_3},L^1_{\eta_1\check{\eta}_2})\quad \mbox{with equivalent norms.}
\eeqs
\item[$(ii)$] For any $p_1,p_2\in(1,\infty)$ satisfying $p_1^{-1}+p_2^{-1}\leq 1$ and any $p_3\in[1,\infty)$,
\beqs
\Modsq{M^{p_1,p_1}_{\eta_1\otimes\eta_1}\hat{\otimes}_{\epsilon} M^{p_3,p_2}_{\eta_3\otimes\eta_2}}= W(\mathcal{F}L^{p_3}_{\check{\eta}_3},L^{\infty}_{\eta_1\check{\eta}_2,0})\quad \mbox{with equivalent norms.}
\eeqs
\end{itemize}
\end{corollary}

\begin{proof} In view of \eqref{stfte-for-m} and Lemma \ref{ident-mod-amalg}, we infer
\beqs
\mathcal{F}^{-1}M^{p_3,p_2}_{\eta_3\otimes\eta_2}= \mathcal{F}M^{p_3,p_2}_{\check{\eta}_3\otimes\check{\eta}_2}= W(\mathcal{F}L^{p_3}_{\check{\eta}_3},L^{p_2}_{\check{\eta}_2})
\eeqs
and consequently $M^{p_3,p_2}_{\eta_3\otimes\eta_2}= \mathcal{F}W(\mathcal{F}L^{p_3}_{\check{\eta}_3},L^{p_2}_{\check{\eta}_2})$. On the other hand, since $\eta_1=\check{\eta}_1$, \cite[Theorem 11.3.5 $(c)$, p. 236]{Grochenig} together with Lemma \ref{ident-mod-amalg} imply $M^{p_1,p_1}_{\eta_1\otimes\eta_1}=W(\mathcal{F}L^{p_1}_{\eta_1},L^{p_1}_{\eta_1})$. Now $(i)$ follows from Theorem \ref{1.21}. The proof of $(ii)$ is analogous and we omit it; of course, in this case, one employs Theorem \ref{meinproposition1} instead.
\end{proof}

Let $\mathcal{Q}_s(\RR^d)$ be the Shubin-Sobolev space of order $s\in\RR$ (see \cite[Chapter 4 Section 25]{Shubin}). Denoting $v_s(x,\xi)=(1+|(x,\xi)|)^s$, $x,\xi\in\RR^d$, $s\in\RR$, \cite[Lemma 2.3]{bog-cor-gro} implies $\mathcal{Q}_s=M^{2,2}_{v_s}$ with equivalent norms. In the following result, we employ the same symbol $v_s$, $s\in\RR$, also for the polynomially bounded weight $x\mapsto(1+|x|)^s$ on $\RR^d$.

\begin{corollary}
For any $s\geq 0$, it holds that
\beq\label{inc-s-tvr}
\Modsq{\mathcal{Q}_s\hat{\otimes}_{\pi} \mathcal{Q}_s}\subseteq W(L^2_{v_s},L^1_{v_s})\cap W(\mathcal{F}L^2_{v_s},L^1_{v_s}),
\eeq
with $v_s(x)=(1+|x|)^s$, $s\in\RR$. In particular,
\beq\label{in-sob-spa-mod-aml}
\Modsq{\mathcal{Q}_s\hat{\otimes}_{\pi} \mathcal{Q}_s}\hookrightarrow W(\mathcal{Q}_s,L^1_{v_s})\hookrightarrow \mathcal{Q}_s,\quad \mbox{for all}\,\, s\geq 0.
\eeq
Furthermore, for any $s\leq 0$, it holds that
\beq\label{inc-fpr-ott}
W(L^2_{v_s},L^1_{v_s})+ W(\mathcal{F}L^2_{v_s},L^1_{v_s})\subseteq \Modsq{\mathcal{Q}_s\hat{\otimes}_{\pi} \mathcal{Q}_s}.
\eeq
\end{corollary}

\begin{proof} We start by pointing out the following two identities (see \cite[Corollary 7]{fei90} and \cite[Theorem 11.3.5 $(c)$, p. 236]{Grochenig}):
\beq\label{iden-fo-prr}
\mathcal{F}L^2_{v_s}=W(\mathcal{F}L^2_{v_s},L^2),\,\, \forall s\in\RR,\quad \mbox{and}\quad \mathcal{F}\mathcal{Q}_s=\mathcal{Q}_s,\,\, \forall s\in\RR.
\eeq
Assume first $s\geq 0$. Then \cite[Proposition 11.3.1 $(b)$ and $(c)$, p. 232]{Grochenig} implies that $M^{2,2}_{v_s}\hookrightarrow L^2_{v_s}$ and $M^{2,2}_{v_s}\hookrightarrow \mathcal{F}L^2_{v_s}$. Consequently $M^{2,2}_{v_s}\subseteq L^2_{v_s}\cap \mathcal{F}L^2_{v_s}$ and $\|\cdot\|_{L^2_{v_s}}+\|\cdot\|_{\mathcal{F}L^2_{v_s}}\leq C_1\|\cdot\|_{M^{2,2}_{v_s}}$. On the other hand, it is straightforward to check $L^2_{v_s}\cap \mathcal{F}L^2_{v_s}\subseteq M^{2,2}_{v_s}$ and $\|\cdot\|_{M^{2,2}_{v_s}}\leq C_2(\|\cdot\|_{L^2_{v_s}}+\|\cdot\|_{\mathcal{F}L^2_{v_s}})$. Consequently,
\beq\label{eq-for-sob-sp}
\mathcal{Q}_s=L^2_{v_s}\cap \mathcal{F}L^2_{v_s},\quad \mbox{and}\,\,\|\cdot\|_{\mathcal{Q}_s}\,\, \mbox{is equivalent to}\,\, \|\cdot\|_{L^2_{v_s}}+\|\cdot\|_{\mathcal{F}L^2_{v_s}}.
\eeq
Thus,
\beq\label{o-for-sob}
\mathcal{Q}_s\hat{\otimes}_{\pi}\mathcal{Q}_s\subseteq L^2_{v_s}\hat{\otimes}_{\pi}\mathcal{F}L^2_{v_s}= W(L^2,L^2_{v_s})\hat{\otimes}_{\pi}\mathcal{F}W(L^2,L^2_{v_s})
\eeq
and Theorem \ref{1.21} implies $\Modsq{\mathcal{Q}_s\hat{\otimes}_{\pi} \mathcal{Q}_s}\subseteq W(L^2,L^1_{v_s^2})=W(L^2_{v_s},L^1_{v_s})$. On the other hand, \eqref{iden-fo-prr} implies
\beq\label{o-for-sob-sss}
\mathcal{Q}_s\hat{\otimes}_{\pi}\mathcal{Q}_s\subseteq W(L^2,L^2_{v_s})\hat{\otimes}_{\pi}\mathcal{F}W(\mathcal{F}L^2_{v_s},L^2)
\eeq
which, in view of Theorem \ref{1.21}, gives $\Modsq{\mathcal{Q}_s\hat{\otimes}_{\pi} \mathcal{Q}_s}\subseteq W(\mathcal{F}L^2_{v_s},L^1_{v_s})$. Consequently, \eqref{inc-s-tvr} holds true. To verify \eqref{in-sob-spa-mod-aml}, we employ \eqref{eq-for-sob-sp} to deduce the continuous inclusions
\beqs
W(L^2_{v_s},L^1_{v_s})\cap W(\mathcal{F}L^2_{v_s},L^1_{v_s})=W(\mathcal{Q}_s,L^1_{v_s})\subseteq W(\mathcal{Q}_s,L^1)\subseteq \mathcal{Q}_s.
\eeqs
This gives the continuous inclusions $\Modsq{\mathcal{Q}_s\hat{\otimes}_{\pi} \mathcal{Q}_s}\subseteq W(\mathcal{Q}_s,L^1_{v_s})\subseteq  \mathcal{Q}_s$. The density follows from the density of $\SSS(\RR^d)$ as they are TMIB spaces.\\
\indent Assume now $s\leq 0$. By \cite[Proposition 11.3.1 $(b)$ and $(c)$, p. 232]{Grochenig}, we infer $L^2_{v_s}\subseteq M^{2,2}_{v_s}$ and $\mathcal{F}L^2_{v_s}\subseteq M^{2,2}_{v_s}$. Thus, \eqref{o-for-sob} and \eqref{o-for-sob-sss} hold true but with the opposite inclusions (cf. \eqref{iden-fo-prr}) and the proof of \eqref{inc-fpr-ott} goes along the same lines as the proof of the first part.
\end{proof}

\begin{remark}
When $s=0$, $\mathcal{Q}_0=L^2$. Hence \eqref{inc-s-tvr} and \eqref{inc-fpr-ott} imply $\Modsq{L^2\hat{\otimes}_{\pi}L^2}=W(L^2,L^1)$, which is a special case of Corollary \ref{cor12} $(a)$.
\end{remark}

Finally, we collect several interesting special instances of the main results of the article.

\begin{remark}
\label{diss}
Throughout what follows $E$ and $F$ are TMIB spaces on $\RR^d$ such that $\omega_F(x)=1$, $\forall x\in\RR^d$, and $\eta$ a polynomially bounded weight on $\RR^d$.
\begin{itemize}
\item[$(i)$] Since $\mathcal{F}W(\mathcal{F}^{-1}E,L^1)\subseteq E\subseteq \mathcal{F}W(\mathcal{F}^{-1}E,L^{\infty}_0)$ (by \eqref{inc-am-sp-n}), Theorem \ref{1.21} and Theorem \ref{meinproposition1} imply
    \beqs
    \Modsq{W(F,L^1_{\eta})\hat{\otimes}_{\pi} E} =W(\mathcal{F}^{-1}E,L^1_{\eta})\quad \mbox{and}\quad \Modsq{W(F,L^{\infty}_{\eta,0})\hat{\otimes}_{\varepsilon} E} =W(\mathcal{F}^{-1}E,L^{\infty}_{\eta,0}).
    \eeqs
    Taking $F=L^1$ and $F=\mathcal{C}_0$ respectively, we deduce (as $W(\mathcal{C}_0,L^{\infty}_{\eta,0})=\mathcal{C}_{\eta,0}$)
    \beq\label{t-m-spa-sdd-ktr}
    \Modsq{L^1_{\eta}\hat{\otimes}_{\pi} E} =W(\mathcal{F}^{-1}E,L^1_{\eta})\quad \mbox{and}\quad \Modsq{\mathcal{C}_{\eta,0}\hat{\otimes}_{\varepsilon} E} =W(\mathcal{F}^{-1}E,L^{\infty}_{\eta,0}).
    \eeq
    Specialising further $E=L^p_{\eta_1}$, $p\in[1,\infty)$, with $\eta_1$ a polynomially bounded weight on $\RR^d$, in view of Lemma \ref{ident-mod-amalg} and the fact $\mathcal{F}^{-1}L^p_{\eta_1}=\mathcal{F}L^p_{\check{\eta}_1}$ with $\|\cdot\|_{\mathcal{F}^{-1}L^p_{\eta_1}}=\|\cdot\|_{\mathcal{F}L^p_{\check{\eta}_1}}$, we infer
    \beq
    \Modsq{L^1_{\eta}\hat{\otimes}_{\pi} L^p_{\eta_1}}= \mathcal{F}M^{p,1}_{\check{\eta}_1\otimes\eta}= \mathcal{F}\left(\Modsq{L^p_{\check{\eta}_1}\hat{\otimes}_{\pi} L^1_{\eta}}\right),
    \eeq
    where, the very last equality follows from the fact that $L^p_{\check{\eta}_1}\hat{\otimes}_{\pi}L^1_{\eta}$ is the Bochner space $L^1_{\eta}(\RR^d_{\xi};L^p_{\check{\eta}_1}(\RR^d_x))$ (see \cite[Section 2.3]{ryan}) and the latter is easily seen to be $L^{p,1}_{\check{\eta}_1\otimes \eta}(\RR^{2d})$. This is interesting as it identifies the Fourier transform of $M^{p,1}_{\check{\eta}_1\otimes\eta}$ as a modulation space. It is known that the Fourier transform does not switch the indexes in $M^{p,1}$ to $M^{1,p}$; the above shows that ``a switching'' does occur but in a more subtle way. Finally, we point out that $\Modsq{L^1_{\eta}\hat{\otimes}_{\pi} L^p_{\eta_1}}$ should not be confused with $M^{1,p}_{\eta\otimes\eta_1}$: in fact $L^1_{\eta}\hat{\otimes}_{\pi} L^p_{\eta_1}$ is the Bochner space $L^1_{\eta}(\RR^d_x;L^p_{\eta_1}(\RR^d_{\xi}))$ but the latter is not $L^{1,p}_{\eta\otimes\eta_1}(\RR^{2d})$.
\item[$(ii)$] Since $W(F,L^1)\subseteq F\subseteq W(F,L^{\infty}_0)$ (by \eqref{inc-am-sp-n}), Theorem \ref{1.21} and Theorem \ref{meinproposition1} imply
    \beqs
    \Modsq{F\hat{\otimes}_{\pi} \mathcal{F}W(E,L^1_{\eta})}=W(E,L^1_{\eta})\quad \mbox{and}\quad \Modsq{F\hat{\otimes}_{\epsilon} \mathcal{F}W(E,L^{\infty}_{\eta,0})}=W(E,L^{\infty}_{\eta,0}).
    \eeqs
    Taking $E=L^1$ and $E=\mathcal{C}_0$ respectively, we deduce (as $W(\mathcal{C}_0,L^{\infty}_{\eta,0})=\mathcal{C}_{\eta,0}$)
    \beqs
    \Modsq{F\hat{\otimes}_{\pi} \mathcal{F}L^1_{\eta}}=L^1_{\eta}\quad \mbox{and}\quad \Modsq{F\hat{\otimes}_{\epsilon} \mathcal{F}\mathcal{C}_{\eta,0}}=\mathcal{C}_{\eta,0}.
    \eeqs
    Consequently, since the growth of the translation group of $\mathcal{F}L^p_{\eta_1}$, $p\in[1,\infty)$, and $\mathcal{F}\mathcal{C}_{\eta_1,0}$ is $1$ for any polynomially bounded weight $\eta_1$, we have
    \begin{align*}
    &\Modsq{L^p\hat{\otimes}_{\pi} \mathcal{F}L^1_{\eta}}=\Modsq{\mathcal{C}_0\hat{\otimes}_{\pi} \mathcal{F}L^1_{\eta}}=\Modsq{\mathcal{F}L^p_{\eta_1}\hat{\otimes}_{\pi} \mathcal{F}L^1_{\eta}} = \Modsq{\mathcal{F}\mathcal{C}_{\eta_1,0}\hat{\otimes}_{\pi} \mathcal{F}L^1_{\eta}} =L^1_{\eta},\,\,\, \forall p\in[1,\infty);\\
    &\Modsq{L^p\hat{\otimes}_{\epsilon} \mathcal{F}L^1_{\eta}}=\Modsq{\mathcal{C}_0\hat{\otimes}_{\epsilon} \mathcal{F}L^1_{\eta}}=\Modsq{\mathcal{F}L^p_{\eta_1}\hat{\otimes}_{\epsilon} \mathcal{F}L^1_{\eta}} = \Modsq{\mathcal{F}\mathcal{C}_{\eta_1,0}\hat{\otimes}_{\epsilon} \mathcal{F}L^1_{\eta}} =\mathcal{C}_{\eta,0},\,\,\, \forall p\in[1,\infty).
    \end{align*}
\item[$(iii)$] Let $\eta_1$ and $\eta_2$ be polynomially bounded weights. Taking $F=L^{p_1}$ and $E=L^{p_2}$ in Theorem \ref{1.21} $(i)$ and Theorem \ref{meinproposition1} $(i)$, we deduce
    \begin{align*}
    &\Modsq{L^{p_1}_{\eta_1}\hat{\otimes}_{\pi}\mathcal{F}L^{p_2}_{\eta_2}}= W(L^{p_2},L^1_{\eta_1\eta_2}),\,\, \mbox{for all}\,\, p_1,p_2\in[1,\infty)\,\, \mbox{satisfying}\,\, p_1^{-1}+p_2^{-1}\geq 1;\\
    &\Modsq{L^{p_1}_{\eta_1}\hat{\otimes}_{\epsilon}\mathcal{F}L^{p_2}_{\eta_2}}= W(L^{p_2},L^{\infty}_{\eta_1\eta_2,0}),\,\, \mbox{for all}\,\, p_1,p_2\in(1,\infty)\,\, \mbox{satisfying}\,\, p_1^{-1}+p_2^{-1}\leq 1.
    \end{align*}
    In particular,
    \beqs
    \Modsq{L^p_{\eta}\hat{\otimes}_{\pi} L^2}=W(L^2,L^1_{\eta}),\, \forall p\in[1,2],\quad \mbox{and}\quad \Modsq{L^p_{\eta}\hat{\otimes}_{\epsilon} L^2}=W(L^2,L^{\infty}_{\eta,0}),\, \forall p\in[2,\infty).
    \eeqs
\end{itemize}
\end{remark}

\begin{remark}\label{rem-for-inbet-spa}
In view of \eqref{t-m-spa-sdd-ktr}, every TMIB space $E$ lies in between two modulations spaces which are subsets of $E_{\operatorname{loc}}$; namely (cf. \eqref{inc-am-sp-n})
\beq\label{set-in-c}
W(E,L^1)\hookrightarrow E\hookrightarrow W(E,L^{\infty}_0)\subseteq E_{\operatorname{loc}}.
\eeq
Additionally, every DTMIB space $F$ also lies between two modulation spaces which are subsets of $F_{\operatorname{loc}}$. To see this, set $F=E'$ , with $E$ a TMIB space. Then \eqref{set-in-c} together with Lemma \ref{dual-ams-s1} $(iii)$ imply
\beq
W(F,L^1)\subseteq F\subseteq W(F,L^{\infty})\subseteq F_{\operatorname{loc}}
\eeq
and \cite[Theorem 4.8 (iii)]{dppv-3} yields that $W(F,L^1)$ and $W(F,L^{\infty})$ are modulation spaces (associated to DTMIB spaces) as they are the strong duals of the modulation spaces $W(E,L^{\infty}_0)$ and $W(E,L^1)$ respectively.
\end{remark}

\begin{remark}
In view of Remark \ref{rem-for-inbet-spa}, it is an interesting question whether every TMIB space can be represented as a modulation space. More precisely, given a TMIB space $E$ on $\RR^d$, does there exist a (D)TMIB space $X$ on $\RR^{2d}$ such that $E=\Modsq{X}$? If this is the case, then there always exists a TMIB space $X_1$ on $\RR^{2d}$ such that $E=\Modsq{X_1}$: it suffices to take $X_1$ to be the closure of $\SSS(\RR^{2d})$ in $X$. Thus, if every TMIB space is a modulation spaces, \cite[Theorem 4.8 (iii)]{dppv-3} implies that every DTMIB is also a modulation space (of course, associated to a DTMIB space).
\end{remark}

\end{document}